\documentclass{amsart}
\usepackage{amsfonts}
\usepackage{calc}
\usepackage{tikz}
\usepackage{pgfplots}
\usepackage{amsmath,amssymb}
\usepackage[foot]{amsaddr}
\usepackage{graphicx,cancel,xcolor,hyperref,comment,graphicx,geometry}
\usepackage{tikz}
\usepackage{graphicx,url,etoolbox}
\usepackage{lipsum}
\usepackage{dsfont}
\usepackage{subfigure}
\usepackage{fancyhdr}
\usepackage{lastpage}
\usepackage{comment}
\usepackage{tikz}
\usepackage{caption}

\usetikzlibrary{shapes}
\let\oldtocsection=\tocsection
\let\oldtocsubsection=\tocsubsection 
\renewcommand{\tocsection}[2]{\vspace{0.5em}\hspace{0em}\oldtocsection{#1}{#2}}
\renewcommand{\tocsubsection}[2]{\vspace{0.5em}\hspace{1em}\oldtocsubsection{#1}{#2}}

\usetikzlibrary{hobby,patterns}
\newtheorem{theoreme}{Theorem}[section]
\newtheorem{pro}[theoreme]{Proposition}
\newtheorem{lemma}[theoreme]{Lemma}
\newtheorem{rem}[theoreme]{Remark}

\theoremstyle{definition}
\numberwithin{equation}{section}
 \renewenvironment{proof}{{\bfseries \noindent Proof.}}{\demo}
\newcommand\xqed[1]{  \leavevmode\unskip\penalty9999 \hbox{}\nobreak\hfill
  \quad\hbox{#1}}
\newcommand\demo{\xqed{$\square$}}
\hypersetup{bookmarks, colorlinks, urlcolor=blue, citecolor=red, linkcolor=blue, hyperfigures, pagebackref,
    pdfcreator=LaTeX, breaklinks=true, pdfpagelayout=SinglePage, bookmarksopen=true,bookmarksopenlevel=2}
\def\u2{\u^2}
\def\u3{\u^3}
\def\u4{\u^4}
\def\u5{\u^5}
\def\y1{\y^1}
\def\y2{\y^2}
\def\y3{\y^3}
\def\y4{\y^4}
\def\y5{\y^5}

\def\la {{\lambda}}

\newcommand {\nc}   {\newcommand}
\nc {\be}   {\begin{equation}} \nc {\ee}   {\end{equation}} \nc
{\beq}  {\begin{eqnarray}} \nc {\eeq}  {\end{eqnarray}} \nc {\beqs}
{\begin{eqnarray*}} \nc {\eeqs} {\end{eqnarray*}}
\def\edc{\end{document}}

\providecommand{\abs}[1]{\lvert#1\rvert}

\usetikzlibrary{decorations.pathmorphing,patterns,scopes,intersections,calc}
 
\numberwithin{dummy}{section}

\numberwithin{equation}{section} 

\begin{document}
\title[\fontsize{7}{9}\selectfont  ]{Stabilization results of a Lorenz
piezoelectric beam with partial viscous dampings }
\author{Mohammad Akil$^{1}$ , Abdelaziz Soufyane$^2$ and Youssef Belhamadia$%
^3$\vspace{0.5cm}}
\address{$^1$Univ. Polytechnique Hauts-de-France, INSA Hauts-de-France,
CERAMATHS-Laboratoire de Mat\'eriaux C\'eramiques et de Math\'ematiques,
F-59313 Valenciennes, France\vspace{0.25cm}\\
$^{2}$ Department of Mathematics, College of Science, University of Sharjah,
P.O.Box 27272, Sharjah, UAE.\vspace{0.25cm}\\
$^{3}$ \ Department of Mathematics and Statistics, American University of
Sharjah. Sharjah, UAE.\\
mohammad.akil@uphf.fr, asoufyane@sharjah.ac.ae , ybelhamadia@aus.edu}

\begin{abstract}
In this paper, we investigate the stabilization of a one-dimensional Lorenz
piezoelectric (Stretching system) with partial viscous dampings. First, by
using Lorenz gauge conditions, we reformulate our system to achieve the
existence and uniqueness of the solution. Next, by using General criteria of
Arendt-Batty, we prove the strong stability in different cases. Finally, we
prove that it is sufficient to control the stretching of the center-line of
the beam in $x-$direction to achieve the exponential stability. Numerical
results are also presented to validate our theoretical result.
\end{abstract}

\maketitle
\tableofcontents


%\keywords{coupled wave equations; viscous damping; $C_0$-semigroup; polynomial stability; cylindrical domains.}
%\date{}
%%%%%%%%%%%%%%%%%%%%%%%%%%%%%%%%%%%%%%%%%%%%%%%%%%%%%%%%%%%%
%Keyword
%%%%%%%%%%%%%%%%%%%%%%%%%%%%%%%%%%%%%%%%%%%%%%%%%%%%%%%%%%%%
%\keywords{}
%%%%%%%%%%%%%%%%%%%%%%%%%%%%%%%%%%%%%%%%%%%%%%%%%%%%%%%%%%%%
%Abstract
%%%%%%%%%%%%%%%%%%%%%%%%%%%%%%%%%%%%%%%%%%%%%%%%%%%%%%%%%%%%

%\pagenumbering{roman}
%\maketitle
%\tableofcontents
%\clearpage
%\pagenumbering{arabic}
%\setcounter{page}{1}
%\setcounter{equation}[section]
\setcounter{equation}{0} %\abstractname{.}

\pagenumbering{roman} %\clearpage
\pagenumbering{arabic} \setcounter{page}{1} %\newpage

\section{Introduction}

\noindent Piezoelectric materials have become more promising in aeronautic,
civil and space structures. It is known, since the 19th century that
materials such as quartz, Rochelle salt and barium titanate under pressure
produces electric charge\slash voltage, this phenomenon is called the direct
piezoelectric effect and was discovered by brothers Pierre and Jacques Curie
in 1880. This same materials, when subjected to an electric field, produce
proportional geometric tension. Such a phenomenon is known as the converse
piezoelectric effect and was discovered by Gabriel Lippmann in 1881 \cite%
{Rogacheva, Tiersten,JiashiYang}. In many studies related to piezoelectric
structures, the magnetic effect is neglected and only the mechanical the
mechanical effects are considered. In general the mechanical effects are
modelled by using Kirchhoff, Euler-Bernoulli or Midlin-Timoshenko
assumptions for small displacements \cite%
{Banks1996SmartMS,757931,RCSmith,JiashiYang}, and electrical and magnetic
effects are added to the system generally using electrostatic, quasi-static
and fully dynamic approaches (\cite%
{Destuynder1992TheoreticalNA,doi:10.1137/050629884,LASIECKA2009167,H.Tzou}).
Morris and \"Ozer in \cite{Morris-Ozer2013,Morris-Ozer2014}, proposed a
variational approach, a piezoelectric beam model with a magnetic effect,
based on the Euler-Bernoulli and Rayleigh beam theory for small
displacement. They considered an elastic beam covered by a piezoelectric
material on its upper and lower surfaces, isolated at the edges and
connected to an external electrical circuit to feed charge to the
electrodes. It is worth mentioning that it is well known that piezoelectric
beams without the magnetic effect, in which they are represented by a wave
equation \cite{Morris-Ozer2014}, are exactly observable \cite%
{doi:10.1137/1030001} and exponentially stable \cite{Tebou2006}. Also, there
exists a few results on piezoelectric material with different kind of
dampings \cite%
{Abdelaziz1,Abdelaziz2,Ramos2018,Ramos2019,zer2019PotentialFF,refId0,Akil-piezo,AnLiuKong}%
.

\noindent Recently, in \cite{OZER-LORENZ2021}, a nouvel infinite-dimensional
models, by a through variational approach, are introduced to describe
vibrations on a piezoelectric beam. Electro-Magnetic effects due to
Maxwell's equations factor in the models via the electric and magnetic
potentials. This system is described by

\begin{equation}  \label{Ozer}
\left\{%
\begin{array}{ll}
\rho v_{tt}-\alpha v_{xx}-\gamma (\phi+\eta_{t})_x=0, & \quad (x,t)\in
(0,L)\times (0,\infty), \\[0.1in] 
\displaystyle -\xi\left(\phi_x+\theta_{t}\right)_x+(\eta_t+\phi)-\frac{\gamma%
}{\varepsilon_3}v_x=0, & \quad (x,t)\in (0,L)\times (0,\infty), \\[0.1in] 
\displaystyle \left(\theta_t+\phi_x\right)_t-\frac{\mu}{\xi\varepsilon_3}%
\left(\eta_x-\theta\right)=\frac{i_s(t)}{\xi\varepsilon_3h}, & \quad
(x,t)\in (0,L)\times (0,\infty), \\[0.1in] 
\displaystyle \left(\eta_t+\phi\right)_t-\frac{\mu}{\varepsilon_3}%
\left(\eta_x-\theta\right)_x-\frac{\gamma}{\varepsilon_3}v_{tx}=0, & \quad
(x,t)\in (0,L)\times (0,\infty), \\[0.1in] 
\displaystyle v(0,t)=\alpha v_x(L,t)+\gamma \phi(L,t)+\gamma \eta_t(L,t)=0,
& \quad t\in (0,\infty), \\[0.1in] 
\xi\varepsilon_3\left(\theta_t+\phi_x\right)(0,t)=\xi\varepsilon_3\left(%
\theta_t+\phi_x\right)(L,t)=0, & \quad t\in (0,\infty), \\[0.1in] 
\mu\left(\theta-\eta_x\right)(0,t)=\mu\left(\theta-\eta_x\right)(L,t)=0 & 
\end{array}
\right.
\end{equation}
where $v,\theta_t+\phi_x,\eta_t+\phi$ and $\theta-\eta_x$ represents
respectively, the stretching of the centreline of the beam in $x-$direction,
electrical field component in $x-$direction, electrical field component in $%
z-$direction and magnetic field component in $y-$direction and $\xi=\frac{%
\varepsilon_1h^2}{12\varepsilon_3}$. The natural physical constants $\rho$, $%
\alpha$, $\gamma$, $\varepsilon_1$, $\varepsilon_3$, $\mu$ denotes the mass
density per unit volume, elastic stifness, piezoelectric coupling
coefficient, permittivity in $x$ and $z$ directions, and magnetic
permeability respectively. The conditions $\eqref{Ozer}_5$-$\eqref{Ozer}_7$,
represents respectively beam clamped on the left, Lateral force, First
charge moment, Current. The applied current $i_s(t)$ at the electrodes
effects only the stretching motion and the surface electrical continuity is
satisfied 
\begin{equation*}
\frac{d i_s(x,t)}{dx}=0. 
\end{equation*}
The author proved that this model fail to be asymptotically stable if the
material parameters satisfy certain conditions. To achieve at least
asymptotic stability the author proposed an additional controller. In this
paper, we study system \eqref{Ozer} without current acting on the electrode
and with different partial viscous damping acting on the stretching of the
centreline of the beam in $x-$direction, electrical field component in $x-$%
direction, electrical field component in $z-$direction and magnetic field
component in $y-$direction. This system in described by 
\begin{equation}  \label{Stretching}
\left\{%
\begin{array}{ll}
\rho v_{tt}-\alpha v_{xx}-\gamma (\phi+\eta_{t})_x+a v_t=0, & \quad (x,t)\in
(0,L)\times (0,\infty), \\[0.1in] 
\displaystyle -\xi\left(\phi_x+\theta_{t}\right)_x+(\eta_t+\phi)-\frac{\gamma%
}{\varepsilon_3}v_x=0, & \quad (x,t)\in (0,L)\times (0,\infty), \\[0.1in] 
\displaystyle \left(\theta_t+\phi_x\right)_t-\frac{\mu}{\xi\varepsilon_3}%
\left(\eta_x-\theta\right)+b\left(\theta_t+\phi_x\right)=0, & \quad (x,t)\in
(0,L)\times (0,\infty), \\[0.1in] 
\displaystyle \left(\eta_t+\phi\right)_t-\frac{\mu}{\varepsilon_3}%
\left(\eta_x-\theta\right)_x-\frac{\gamma}{\varepsilon_3}v_{tx}+c\left(%
\eta_t+\phi\right)=0, & \quad (x,t)\in (0,L)\times (0,\infty), \\[0.1in] 
\displaystyle v(0,t)=\alpha v_x(L,t)+\gamma \phi(L,t)+\gamma \eta_t(L,t)=0,
& \quad t\in (0,\infty), \\[0.1in] 
\xi\varepsilon_3\left(\theta_t+\phi_x\right)(0,t)=\xi\varepsilon_3\left(%
\theta_t+\phi_x\right)(L,t)=0, & \quad t\in (0,\infty), \\[0.1in] 
\mu\left(\theta-\eta_x\right)(0,t)=\mu\left(\theta-\eta_x\right)(L,t)=0 & 
\end{array}
\right.  \tag{${\rm Stretching}$}
\end{equation}
where $a,b,c>0$. In the first section we reformulate and we prove the
well-posedness of our system. In the second part we prove the strong
stability of system \eqref{Stretching}. Next, we prove the exponential
stability under partial viscous damping on the centreline of the beam in $x-$%
direction and$\backslash$or electrical field component in ($x$ and $z$%
)-direction. Finally, we numerically illustrate the exponential stability
decay of the natural energy $E(t)$ of \eqref{Lorenz} system. 
%%%%%%%%%%%%%%%%%%%%%%%%%%%%%%%%%%%%%%%%%%%%%%%%%%%%%%%%%
%%% Wellposedness 
%%%%%%%%%%%%%%%%%%%%%%%%%%%%%%%%%%%%%%%%%%%%%%%%%%%%%%%%%%

\section{Reformulation and Wellposedness}

\noindent System \eqref{Stretching} does not yield a unique solution since:

\begin{enumerate}
\item[$\bullet$] The magnetic potential vector component $\theta, \eta$ and
the electrical potential $\phi$ are not uniquely defined (see Equation (1)
in \cite{OZER-LORENZ2021} and \cite{refId0} .

\item[$\bullet$] The Lagrangian is invariant under certain transformations 
\cite{refId0}
\end{enumerate}

To obtain a unique solution, particular gauge conditions are presented in
electro-magnetic theory to completely decouple the electromagnetic equations
in \eqref{Stretching}. One of the most widely used gauges is Lorenz Gauges 
\cite%
{Morris-Ozer2014,OZER-LORENZ2021,refId0,Tataru2016,doi:10.1080/03605301003717100}%
). For the piezoelectric beam model, the Lorenz Gauge condition is given by 
\begin{equation}  \label{LGC}
-\xi \theta_x+\eta=\frac{\xi \varepsilon_3}{\mu}\phi_t, 
\tag{${\rm
\text{LGC}}$}
\end{equation}
with the boundary conditions 
\begin{equation}  \label{bctheta}
\theta(0,t)=\theta(L,t)=0.
\end{equation}
In the casse of \eqref{LGC}, the term $-\xi \theta_{tx}+\eta_t$ in $%
\eqref{Stretching}_2$ is transformed into $\frac{\xi\varepsilon_3}{\mu}%
\phi_{tt}$. As well, the terms $\phi_{tx}-\frac{\mu}{\xi\varepsilon_3}%
(\eta_x-\theta)$ and $\phi_t-\frac{\mu}{\varepsilon_3}(\eta_x-\theta)_x$ in $%
\eqref{Stretching}_3$ and $\eqref{Stretching}_4$ are transformed into $-%
\frac{\mu}{\xi \varepsilon_3}(\xi\theta_{xx}-\theta)$ and $-\frac{\mu}{\xi
\varepsilon_3}(\xi \eta_{xx}-\eta)$, respectively. This transformation not
only the $\phi-$equation to a wave equation but also the $\theta$ and $\eta$
equations. Therefore, both electric and magnetic equations are wave
equations. Then, the equations of motion \eqref{Stretching}-\eqref{bctheta}
respectively reduce to 
\begin{equation}  \label{Lorenz}
\left\{%
\begin{array}{ll}
\rho v_{tt}-\alpha v_{xx}-\gamma (\phi+\eta_{t})_x+a v_t=0, & \quad (x,t)\in
(0,L)\times (0,\infty), \\[0.1in] 
\displaystyle \phi_{tt}-\frac{\mu}{\varepsilon_3}\phi_{xx}+\frac{\mu}{%
\xi\varepsilon_3}\phi-\frac{\gamma \mu}{\xi\varepsilon_3^2}v_x=0, & \quad
(x,t)\in (0,L)\times (0,\infty), \\[0.1in] 
\displaystyle \theta_{tt}-\frac{\mu}{\varepsilon_3}\theta_{xx}+\frac{\mu}{%
\xi\varepsilon_3}\theta +b(\theta_t+\phi_x)=0, & \quad (x,t)\in (0,L)\times
(0,\infty), \\[0.1in] 
\eta_{tt}-\frac{\mu}{\varepsilon_3}\eta_{xx}+\frac{\mu}{\xi\varepsilon_3}%
\eta-\frac{\gamma}{\varepsilon_3}v_{tx}+ c(\eta_t+\phi)=0, & \quad (x,t)\in
(0,L)\times (0,\infty), \\[0.1in] 
\displaystyle v(0,t)=\alpha v_x(L,t)+\gamma \phi(L,t)+\gamma \eta_t(L,t)=0,
& \quad t\in (0,\infty), \\[0.1in] 
\phi_x(0,t)=\phi_x(L,t)=\eta_x(0,t)=\eta_x(L,t)=\theta(0,t)=\theta(L,t)=0, & \quad t\in (0,\infty),
\\[0.1in] 
(v,\phi,\eta,v_t,\phi_t,\eta_t)(\cdot,0)=(v^0,0,\eta^0,v^1,0,\eta^1), & 
\quad x\in (0,L).%
\end{array}
\right.  \tag{$${\rm Lorenz}}
\end{equation}

\begin{lemma}
\label{Energy} The natural energy $E(t)$ associated to \eqref{Lorenz} system
is the sum of Kinetic, potential, magnetic and electrical energies, i,e, 
\begin{equation}  \label{ENERGY1}
E(t)=E_k(t)+E_p(t)+E_{B}(t)+E_{elec}(t),
\end{equation}
where 
\begin{equation*}
\left\{%
\begin{array}{l}
\displaystyle E_k(t)=\frac{\rho}{2}\int_0^L\abs{v_t}^2dx,\ E_p(t)=\frac{%
\alpha}{2}\int_0^L\abs{v_x}^2dx,\ E_{B}(t)=\frac{\mu}{2}\int_0^L%
\abs{\theta-\eta_x}^2dx, \\[0.1in] 
\displaystyle E_{elec}(t)=\frac{1}{2}\int_0^L\left[\xi\varepsilon_3%
\abs{\theta_t+\phi_x}^2+\varepsilon_3\abs{\eta_t+\phi}^2\right]dx%
\end{array}
\right.
\end{equation*}
and 
\begin{equation}  \label{denergy}
\frac{d}{dt}E(t)=-a\int_0^L\abs{v_t}^2-b\xi \varepsilon_3\int_0^L%
\abs{\theta_t+\phi_x}^2dx-c\varepsilon_3\int_0^L\abs{\eta_t+\phi}^2dx.
\end{equation}
\end{lemma}

\begin{proof}
Multiplying $\eqref{Lorenz}_1$ by $\overline{v_t}$, integrating by parts
over $(0,L)$ and taking the real part, we get 
\begin{equation}  \label{E1}
\frac{d}{dt}E_k(t)+\frac{d}{dt}E_p(t)+\Re\left(\gamma\int_0^L(\phi+\eta_t)%
\overline{v_{xt}}dx\right)+a\int_0^L\abs{v_t}^2dx=0.
\end{equation}
Multiplying $\eqref{Lorenz}_3$ by $\xi \varepsilon_3\overline{%
\left(\theta_t+\phi_x\right)}$, integrating over $(0,L)$, we get 
\begin{equation}  \label{E2}
\begin{array}{l}
\displaystyle \frac{\xi\varepsilon_3}{2}\frac{d}{dt}\int_0^L\abs{\theta_t}%
^2dx+\frac{\mu}{2}\frac{d}{dt}\int_0^L\abs{\theta}^2dx+\Re\left(\xi%
\varepsilon_3\int_0^L\theta_{tt}\overline{\phi_x}dx\right)+\Re\left(\mu\xi%
\int_0^L\theta_x\overline{\left(\theta_{tx}+{\phi}_{xx}\right)}dx\right) \\%
[0.1in] 
\displaystyle +\Re\left(\mu\int_0^L\theta \overline{\phi_x}dx\right)+b\xi
\varepsilon_3\int_0^L\abs{\theta_t+\phi_x}^2dx=0.%
\end{array}%
\end{equation}
Using \eqref{LGC} in the fourth integral in \eqref{E2}, we obtain 
\begin{equation}  \label{E3}
\begin{array}{l}
\displaystyle \Re\left(\mu\xi\int_0^L\theta_x\overline{\left(\theta_{tx}+{%
\phi}_{xx}\right)}dx\right)=\frac{\xi\varepsilon_3}{2}\frac{d}{dt}\int_0^L%
\abs{\phi_x}^2dx-\Re\left(\mu\int_0^L\eta_x\overline{\theta_t}dx\right) \\%
[0.1in] 
\displaystyle -\Re\left(\mu\int_0^L\eta_x\overline{\phi_x}%
dx\right)+\Re\left(\xi \varepsilon_3\int_0^L\phi_{xt}\overline{\theta_t}%
dx\right)%
\end{array}%
\end{equation}
Inserting \eqref{E3} in \eqref{E2}, we get 
\begin{equation}  \label{E4}
\begin{array}{l}
\displaystyle \frac{\xi\varepsilon_3}{2}\frac{d}{dt}\int_0^L%
\abs{\theta_t+\phi_x}^2dx+\frac{\mu}{2}\frac{d}{dt}\int_0^L\abs{\theta}^2dx
-\Re\left(\mu\int_0^L\eta_x\overline{\theta_t}dx\right)-\Re\left(\mu\int_0^L%
\eta_x\overline{\phi_x}dx\right) \\[0.1in] 
\displaystyle +\Re\left(\mu\int_0^L\theta \overline{\phi_x}dx\right)+b\xi
\varepsilon_3\int_0^L\abs{\theta_t+\phi_x}^2dx=0.%
\end{array}%
\end{equation}
Multiplying $\eqref{Lorenz}_4$ by $\varepsilon_3\overline{(\eta_t+\phi)}$,
integrating by parts over $(0,L)$, we get 
\begin{equation}  \label{E5}
\begin{array}{l}
\displaystyle \frac{\varepsilon_3}{2}\frac{d}{dt}\int_0^L\abs{\eta_t}^2dx+%
\frac{\mu}{2}\frac{d}{dt}\int_0^L\abs{\eta_x}^2dx+\Re\left(\varepsilon_3%
\int_0^L\eta_{tt}\overline{\phi}dx\right)+\Re\left(\mu\int_0^L\eta_x%
\overline{\phi_x}dx\right) \\[0.1in] 
\displaystyle +\Re\left(\frac{\mu}{\xi}\int_0^L\eta\overline{(\eta_t+\phi)}%
dx\right)-\Re\left(\gamma\int_0^Lv_{xt}\overline{(\phi+\eta_t)}%
dx\right)+c\varepsilon_3\int_0^L\abs{\eta_t+\phi}^2dx=0.%
\end{array}%
\end{equation}
Using \eqref{LGC} in the fifth integral over \eqref{E5} and integrating by
parts over $(0,L)$, we obtain 
\begin{equation}  \label{E6}
\begin{array}{l}
\displaystyle \Re\left(\frac{\mu}{\xi}\int_0^L\eta\overline{(\eta_t+\phi)}%
dx\right)=\frac{\varepsilon_3}{2}\frac{d}{dt}\int_0^L\abs{\phi}%
^2dx+\Re\left(\varepsilon_3\int_0^L\phi_t\overline{\eta_t}dx\right) \\%
[0.1in] 
\displaystyle -\Re\left(\mu\int_0^L\theta\overline{\eta_{xt}}%
dx\right)-\Re\left(\mu\int_0^L\theta\overline{\phi_x}dx\right).%
\end{array}%
\end{equation}
Inserting \eqref{E6} in \eqref{E5}, we get 
\begin{equation}  \label{E7}
\begin{array}{l}
\displaystyle \frac{\varepsilon_3}{2}\frac{d}{dt}\int_0^L\abs{\eta_t+\phi}%
^2dx+\frac{\mu}{2}\frac{d}{dt}\int_0^L\abs{\eta_x}^2dx-\Re\left(\mu\int_0^L%
\theta\overline{\eta_{xt}}dx\right)+\Re\left(\mu\int_0^L\eta_x\overline{%
\phi_x}dx\right) \\[0.1in] 
\displaystyle -\Re\left(\mu\int_0^L\theta\overline{\phi_x}%
dx\right)-\Re\left(\gamma\int_0^Lv_{xt}\overline{(\phi+\eta_t)}%
dx\right)+c\varepsilon_3\int_0^L\abs{\eta_t+\phi}^2dx=0.%
\end{array}%
\end{equation}
Adding \eqref{E1}, \eqref{E4} and \eqref{E7}, we get 
\begin{equation*}
\frac{d}{dt}\left(E_k(t)+E_p(t)+E_{B}(t)+E_{elec}(t)\right)=-a\int_0^L%
\abs{v_t}^2-b\xi \varepsilon_3\int_0^L\abs{\theta_t+\phi_x}%
^2dx-c\varepsilon_3\int_0^L\abs{\eta_t+\phi}^2dx.
\end{equation*}
Thus, we obtain \eqref{denergy}. The proof has been completed.
\end{proof}

%%%%%%%%%%%%%%%%%%%%%%%%%%%%%%%%%%%

\begin{lemma}
\label{Norm1} If $E(t)=0$ then $v=\phi=\eta=\theta=0$.
\end{lemma}

\begin{proof}
By \eqref{ENERGY1}, $E(t)=0$ implies that 
\begin{equation}  \label{1Norm1}
v_t(x,t)=v_x(x,t)=\theta(x,t)-\eta_x(x,t)=\theta_t(x,t)+\phi_x(x,t)=%
\eta_t(x,t)+\phi(x,t)=0.
\end{equation}
Using the fact that $v(0,t)=0$ in \eqref{1Norm1}, we get 
\begin{equation}  \label{2Norm1}
v(x,t)=0\quad \text{in}\quad (0,L)\times (0,\infty).
\end{equation}
Using \eqref{2Norm1} and $\eqref{Lorenz}_2$, we get the following system 
\begin{equation}  \label{3Norm1}
\left\{%
\begin{array}{lll}
\phi_{tt}(x,t)-\frac{\mu}{\varepsilon_3}\phi_{xx}(x,t)+\frac{\mu}{%
\xi\varepsilon_3}\phi(x,t)=0 & \text{in} & (0,L)\times (0,\infty), \\ 
\phi_x(0,t)=\phi_x(L,t)=0 & \text{in} & (0,\infty), \\ 
\phi(x,0)=\phi_t(x,0)=0 & \text{in} & (0,L).%
\end{array}
\right.
\end{equation}
Applying Fourrier Transforms with respect to the variable $x$ on %
\eqref{3Norm1} and using the zeros initial conditions, we get 
\begin{equation}  \label{4Norm1}
\phi(x,t)=0.
\end{equation}
Using \eqref{4Norm1} and \eqref{1Norm1}, we get $\theta_t=0$. Using the fact
that $\theta_t=0$ and $\eqref{Lorenz}_3$, we get 
\begin{equation}  \label{5Norm1}
\left\{%
\begin{array}{lll}
\xi\theta_{xx}(x,t)-\theta(x,t)=0 & \text{in} & (0,L)\times (0,\infty), \\ 
\theta(0,t)=\theta(L,t)=0 & \text{in} & (0,\infty).%
\end{array}
\right.
\end{equation}
The solution of \eqref{5Norm1} is $\theta(x,t)=0$. Finally, using \eqref{LGC}
and the fact that $\theta(x,t)=0$, we get $\eta(x,t)=0$. The proof has been
completed.
\end{proof}

%%%%%%%%%%%%%%%%%%%%%%%%%%%%%%%%%%%
%\newline
\noindent Now, we define the following state 
\begin{equation*}
U=\left(v,z,u^1,u^2,u^3\right)
\end{equation*}
such that $z=u_t$, $u^1=\theta-\eta_x$, $u^2=\theta_t+\phi_x$ and $%
u^3=\eta_t+\phi$. with the following initial condition 
\begin{equation*}
U(\cdot,0)=U_0=\left(v(\cdot,0),z(\cdot,0),\theta(\cdot,0)-\eta_x(\cdot,0),%
\theta_t(\cdot,0),\eta_t(\cdot,0)\right)
\end{equation*}
By the choices of the states, $\eqref{Lorenz}_{2}$ and \eqref{LGC}, we
obtain the following compatibility condition: 
\begin{equation}  \label{compatibility}
\xi u^2_x-u^3+\frac{\gamma}{\varepsilon_3}v_x=0.
\end{equation}
We define the linear space 
\begin{equation}  \label{H}
\mathcal{H}=\left\{U\in \left(H_L^1(0,L)\times
\left(L^2(0,L)\right)^4\right), u^2_x\in L^2(0,L),\ u^2(0)=u^2(L)=0,\ \xi
u^2_x-u^3+\frac{\gamma}{\varepsilon_3}v_x=0\right\}
\end{equation}
and the bilinear form on $\mathcal{H}\times \mathcal{H}$ 
\begin{equation}  \label{bilinear}
b(U,\widetilde{U})=\int_0^L\left(\alpha v_x\overline{\widetilde{v}_x}+\rho z%
\overline{\widetilde{z}}+\mu u^1\overline{\widetilde{u^1}}+\xi
\varepsilon_3u^2\overline{\widetilde{u^2}}+\varepsilon_3u^3\overline{%
\widetilde{u^3}}\right)dx.
\end{equation}

\begin{rem}
Using \eqref{compatibility}, the. bilinear form $b$ can be written as 
\begin{equation}  \label{bilinear1}
b(U,\widetilde{U})=\int_0^L\left(\alpha v_x\overline{\widetilde{v}_x}+\rho z%
\overline{\widetilde{z}}+\mu u^1\overline{\widetilde{u^1}}+\xi
\varepsilon_3u^2\overline{\widetilde{u^2}}+\varepsilon_3\left(\xi u^2_x+%
\frac{\gamma}{\varepsilon_3}v_x\right)\left(\xi\overline{\widetilde{u^2}_x}+%
\frac{\gamma}{\varepsilon_3}\overline{\widetilde{v_x}}\right)\right)dx.
\end{equation}
\end{rem}

\begin{lemma}
(See \cite{OZER-LORENZ2021})The bilinear form $b$ is symmetric, continuous
and coercive on $\mathcal{H}\times \mathcal{H}$.
\end{lemma}

\begin{proof}
The bilinear form \eqref{bilinear} is symmetric and by using the
Poincar\'e's inequality on $u^2$ terms, we can check easily the continuity.
For the coercivity, using \eqref{bilinear1} and the generalized Young's
inequality, we get 
\begin{equation*}
b(U,U)\geq \int_0^L\left(\rho \abs{z}^2+\mu\abs{u^1}^2+\xi\varepsilon_3%
\abs{u^2}^2+\left(\alpha+\frac{\gamma^2}{\varepsilon_3}-\frac{\gamma\xi}{k}%
\right)\abs{v_x}^2+\left(\varepsilon_3\xi^2-\gamma\xi k\right)\abs{u^2_x}%
^2\right)dx
\end{equation*}
By choosing 
\begin{equation*}
\frac{\gamma\xi}{\alpha+\frac{\gamma^2}{\varepsilon_3}}<k<\frac{%
\xi\varepsilon_3}{\gamma}
\end{equation*}
then the coefficients of $\abs{v_x}^2$ and $\abs{u^2_x}^2$ are positive.
Therefore, 
\begin{equation*}
b(U,U)\geq C\|U\|_{\mathcal{H}},
\end{equation*}
where $C=\displaystyle{\min\left(\rho,\mu,\xi\varepsilon_3,\left(\alpha+%
\frac{\gamma^2}{\varepsilon_3}-\frac{\gamma\xi}{k}\right),\left(%
\varepsilon_3\xi^2-\gamma\xi k\right)\right)}$. The proof has been completed.
\end{proof}

%%%%%%%%%%

\begin{lemma}
$\mathcal{H}$ is a Hilbert space equipped by the inner product $b(U,%
\widetilde{U})$.
\end{lemma}

\noindent We define the unbounded linear operator $\mathcal{A}:D(\mathcal{A}%
)\longrightarrow \mathcal{H}$, by 
\begin{equation*}
\mathcal{A}U=%
\begin{pmatrix}
z \\[0.1in] 
\displaystyle\frac{\alpha }{\rho }v_{xx}+\frac{\gamma }{\rho }u_{x}^{3} \\%
[0.1in] 
\displaystyle u^{2}-u_{x}^{3} \\[0.1in] 
\displaystyle-\frac{\mu }{\xi \varepsilon _{3}}u^{1} \\[0.1in] 
\displaystyle-\frac{\mu }{\varepsilon _{3}}u_{x}^{1}+\frac{\gamma }{%
\varepsilon _{3}}z_{x}%
\end{pmatrix}%
\end{equation*}%
and 
\begin{equation*}
D(\mathcal{A})=\left\{ 
\begin{array}{l}
U=(v,z,u^{1},u^{2},u^{3})\in \mathcal{H};\ z\in H_{L}^{1}(0,L),v\in
H^{2}(0,L)\cap H_{L}^{1}(0,L),\ u^{1},u^{2}\in H_{0}^{1}(0,L), \\ 
u^{3}\in H^{1}(0,L)\quad \text{and}\quad \alpha v_{x}(L)+\gamma u^{3}(L)=0%
\end{array}%
\right\}
\end{equation*}%
%
%%%%%%%%%%%%%%%%%%%%%%%%%%%%%%%%%%

\begin{pro}
(See \cite{OZER-LORENZ2021}) We have:

\begin{enumerate}
\item $0\in \rho\left(\mathcal{A}\right)$.

\item The operator $\mathcal{A}$ satisfies $\mathcal{A}^{\ast}=-\mathcal{A}$
on $\mathcal{H}$, and $\mathcal{A}$ is a generator of a unitary semigroup $%
\left(e^{t\mathcal{A}}\right)_{t\geq 0}$.
\end{enumerate}
\end{pro}

\noindent The system \eqref{Lorenz} can be written as 
\begin{equation}  \label{Evolution}
U_t=\left(\mathcal{A}-B\right)U,\quad U(0)=U_0.
\end{equation}
where 
\begin{equation*}
BU:=\left(0,\frac{a}{\rho}z,0,bu^2,cu^3\right)^{\top}.
\end{equation*}
It is easy to see that the operator $B$ is a bounded operator. Let us denote 
$\mathcal{A}_{a,b,c}=\mathcal{A}-B$. The operator $\mathcal{A}_{a,b,c}$
defined by \eqref{Evolution} with domain $D(\mathcal{A}_{a,b,c})=D(\mathcal{A%
})$ is densely defined in $\mathcal{H}$. Moreover, $\mathcal{A}_{a,b,c}$ is
the infinitesimal generator of $C_0-$semigroup of contractions. Therefore,
by Lumer-Philips theorem if $U_0\in D(\mathcal{A})$ solution of %
\eqref{Evolution} then $U\in C([0,T];D(\mathcal{A}))\cap C^1([0,T];\mathcal{H%
})$. %%%%%%%%%%%%%%%%%%%%%%%%%%%%

\section{Strong Stability}

\noindent The aim of this section is to analyse the strong stability of
system $\eqref{Evolution}$. The main result of this section is the following
theorems.

\begin{theoreme}
\label{THM-ST-ST1} The $C_0-$semigroup of contractions $(e^{t\mathcal{A}%
_{a,b,c}})_{t\geq 0}$ is strongly stable in $\mathcal{H}$ is the sense that $%
\displaystyle{\lim_{t\to +\infty}\|e^{t\mathcal{A}_{a,b,c}}U_0\|_\mathcal{H}%
=0}$, in the following cases:

\begin{enumerate}
\item[$\mathbf{Case 1:}$] $(a,b,c)\neq (0,0,0)$.

\item[$\mathbf{Case 2:}$] $a=0$ and $(b,c)\neq (0,0)$.

\item[$\mathbf{Case 3:}$] $b=0$ and $(a,c)\neq (0,0)$.

\item[$\mathbf{Case 4:}$] $c=0$ and $(a,b)\neq (0,0)$.

\item[$\mathbf{Case 5:}$] $a\neq 0$ and $(b,c)=(0,0)$.

\item[$\mathbf{Case 6:}$] $b\neq 0$ and $(a,c)=(0,0)$.
\end{enumerate}
\end{theoreme}

\begin{proof}
Since the resolvent of $\mathcal{A}_{a,b,c}$ is compact in $\mathcal{H}$,
then according to Arendt-Batty theorem see (Page 837 in \cite{Arendt01}),
system \eqref{Lorenz} is strongly stable if and only if $\mathcal{A}$
doesn't have pure imaginary eigenvalues, that is, $\sigma(\mathcal{A})\cap i%
\mathbb{R}=\emptyset$. We have already shown that $0\in \rho(\mathcal{A}%
_{a,b,c})$, and still need to show that $\sigma(\mathcal{A}_{a,b,c})\cap i%
\mathbb{R}^{\ast}=\emptyset$. for this aim, suppose by contradiction that
there exists $\lambda\in \mathbb{R}^{\ast}$ and $U\in D(\mathcal{A}%
_{a,b,c})\backslash \{0\}$ such that 
\begin{equation}  \label{KER}
\mathcal{A}_{a,b,c}U=i\lambda U.
\end{equation}
Equivalently, we have: 
\begin{eqnarray}
z&=&i\la v,  \label{KER1} \\[0.1in]
\la^2\rho v+\alpha v_{xx}+\gamma u^3_x-az&=&0,  \label{KER2} \\[0.1in]
u^2-u^3_x&=&i\la u^1,  \label{KER3} \\[0.1in]
-\frac{\mu}{\xi \varepsilon_3}u^1-bu^2&=&i\la u^2,  \label{KER4} \\[0.1in]
-\frac{\mu}{\varepsilon_3}u^1_x+\frac{\gamma}{\varepsilon_3}i\la v_x-cu^3&=&i%
\la u^3.  \label{KER5}
\end{eqnarray}
A straightforward calculation gives: 
\begin{equation*}
0=\Re\left<i\la U,U\right>_{\mathcal{H}}=\Re\left<\mathcal{A}%
_{a,b,c}U,U\right>_{\mathcal{H}}=-a\int_0^L\abs{z}^2dx-b\xi
\varepsilon_3\int_0^L\abs{u^2}^2dx-c\varepsilon_3\int_0^L\abs{u^3}^2dx.
\end{equation*}
Consequently, we deduce that: 
\begin{equation}  \label{Ker-diss}
az=bu^2=cu^3=0.
\end{equation}
\textbf{Case 1}: From \eqref{Ker-diss}, we get $z=u^2=u^3=0$. Using the fact
that $\la \neq 0$, \eqref{KER1} and \eqref{KER4}, we get $v=0$ and $u^1$.
Thus, $U=0$ and consequently $\mathcal{A}$ has no pure imaginary eigenvalues.%
\newline
\textbf{Case 2}: From \eqref{Ker-diss}, we get $u^2=u^3=0$. Then, from %
\eqref{KER4} we obtain 
\begin{equation}  \label{Case2-Eq1}
u^1=0.
\end{equation}
Using \eqref{Ker-diss}, \eqref{Case2-Eq1} and the fact that $\la\neq 0$ in %
\eqref{KER5}, we get $v_x=0$. Using the boundary condition $v(0)=0$ and the
fact that $\la \neq 0$, we obtain $v=0$. Using the fact that that $v=0$ and $%
\la \neq 0$ in \eqref{KER1}, we get $z=0$. Thus, $U=0$ and consequently $%
\mathcal{A}$ has no pure imaginary eigenvalues.\newline
\textbf{Case 3}: From \eqref{Ker-diss}, we get $z=u^3=0$. Then, from %
\eqref{KER1} and the fact that $\la\neq 0$, we get $v=0$. Using the fact
that $v=u^3=0$ and $u^1(0)=0$ in \eqref{KER5}, we get $u^1=0$. Using the
fact that $u^3=u^1=0$ in \eqref{KER3}, we get $u^2=0$. Thus, $U=0$ and
consequently $\mathcal{A}$ has no pure imaginary eigenvalues.\newline
\textbf{Case 4}: From \eqref{Ker-diss}, we get $z=u^2=0$. Then, from %
\eqref{KER1}, \eqref{KER4} and the fact that $\la \neq 0$, we get $v=u^1=0$.
Using $v=u^1=0$ and the fact that $\la \neq 0$ in \eqref{KER5}, we get $%
u^3=0 $. Thus, $U=0$ and consequently $\mathcal{A}$ has no pure imaginary
eigenvalues.\newline
\textbf{Case 5}: From \eqref{Ker-diss}, we get $z=0$. Using the facts that $%
z=0$ and $\la \neq 0$ in \eqref{KER1}, we get $v=0$. Using $z=v=0$ in %
\eqref{KER2}, we obtain $u^3_x=0$, it follows that 
\begin{equation}  \label{Case5-EQ1}
u^3=k.
\end{equation}
Using the fact that $v=0$ and inserting \eqref{Case5-EQ1} in \eqref{KER5},
we get 
\begin{equation}  \label{Case5-EQ2}
-\frac{\mu}{\varepsilon_3}u^1=i\la kx+k_1.
\end{equation}
Using the fact that $u^1(0)=u^1(L)=0$ in \eqref{Case5-EQ2}, we get $k=k_1=0$%
, it follows that 
\begin{equation}  \label{Case5-EQ3}
u^1=u^3=0.
\end{equation}
Inserting \eqref{Case5-EQ3} in \eqref{KER4} and using the fact that $\la %
\neq 0$, we get $u^2=0$. Thus, $U=0$ and consequently $\mathcal{A}$ has no
pure imaginary eigenvalues.\newline
\textbf{Case 6}: From \eqref{Ker-diss}, we get $u^2=0$. Then, from %
\eqref{KER4}, we get $u^1=0$. Using the fact that $u^2=u^1=0$ in \eqref{KER3}%
, we get $u^3_x=0$. Using the fact that $u^1=u_x^3=0$ in \eqref{KER5}, we
get 
\begin{equation}  \label{Case6-EQ1}
\frac{\gamma}{\varepsilon_3}v_x=u^3.
\end{equation}
Deriving the above equation and using the fact that $u^3_x=0$, we obtain 
\begin{equation}  \label{Case6-EQ2}
v_{xx}=0.
\end{equation}
Inserting \eqref{Case6-EQ2} in \eqref{KER2} and using the fact that $\la\neq
0$, we get $v=0$. Then, from \eqref{Case6-EQ1}, we obtain $u^3=0$. Thus, $%
U=0 $ and consequently $\mathcal{A}$ has no pure imaginary eigenvalues.
\end{proof}

%%%%%%%%%%%%%%%%%%%%%%%%%%%%%%%%%%%%%

\begin{theoreme}
\label{THM-ST-ST2} Assume that $c\neq 0$ and $(a,b)=(0,0)$. Then, the $C_0-$%
semigroup of contractions $\left(e^{t\mathcal{A}_{0,0,c}}\right)$ is
strongly stable on $\mathcal{H}$ in the sense that $\displaystyle{%
\lim_{t\to+\infty}\|e^{t\mathcal{A}_{0,0,c}}\|}=0$ for all $U_0\in \mathcal{H%
}$ if and only if 
\begin{equation*}  \label{SC}
{\tag{${\rm SC}$}} \frac{\mu\rho}{\xi\varepsilon_3\alpha}\neq \frac{%
(2n+1)^2\pi^2}{4 L^2}.
\end{equation*}
\end{theoreme}

\begin{proof}
We suppose by contradiction that there exists $\la\in \mathbb{R}^{\ast}$ and 
$U\in D(\mathcal{A}_{0,0,c})\backslash \{0\}$ such that 
\begin{equation}  \label{Th2-Ker0}
\mathcal{A}_{0,0,c}U=i\la U.
\end{equation}
A straightforward computation gives: 
\begin{equation*}
0=\Re\left<i\la U,U\right>_{\mathcal{H}}=\Re\left<\mathcal{A}%
_{0,0,c}U,U\right>_{\mathcal{H}}=-c\varepsilon_3\int_0^L\abs{u^3}^2dx.
\end{equation*}
Consequently, we deduce that 
\begin{equation}  \label{TH2-EQ1}
u^3=0.
\end{equation}
Detailing \eqref{Th2-Ker0} and using \eqref{TH2-EQ1}%
\begin{eqnarray}
z&=&i\la v,  \label{Th2-KER1} \\[0.1in]
\la^2\rho v+\alpha v_{xx}&=&0,  \label{Th2-KER2} \\[0.1in]
u^2&=&i\la u^1,  \label{Th2-KER3} \\[0.1in]
-\frac{\mu}{\xi \varepsilon_3}u^1&=&i\la u^2,  \label{Th2-KER4} \\[0.1in]
-\frac{\mu}{\xi \varepsilon_3}u^1_x+\frac{\gamma}{\varepsilon_3}i\la v_x&=&0.
\label{Th2-KER5}
\end{eqnarray}
Inserting \eqref{Th2-KER3} in \eqref{Th2-KER4}, we get 
\begin{equation}  \label{Th2-KER6}
\left(\la^2-\frac{\mu}{\xi\varepsilon_3}\right)u^1=0.
\end{equation}
We distinguish two cases:\newline
\textbf{Case 1}: If $\la^2\neq \frac{\mu}{\xi \varepsilon_3}$, it follows
that $u^1=0$. Using the fact that $u^1=0$ in \eqref{Th2-KER3}, we obtain $%
u^2=0$. Using the fact that $\la \neq 0$ and $u^1=0$ and $v(0)=0$ in %
\eqref{Th2-KER5}, we get $v=0$ then $z=0$. Thus, $U=0$ and consequently $%
\mathcal{A}$ has no pure imaginary eigenvalues. \newline
\textbf{Case 2:} If $\la^2= \frac{\mu}{\xi \varepsilon_3}$. From, %
\eqref{Th2-KER5} and the fact that $v(0)=u^1(0)$, we get 
\begin{equation}  \label{Th2-Ker7}
u^1=i\la\frac{\gamma}{\mu}v.
\end{equation}
Using \eqref{TH2-EQ1}, the compatibility condition \eqref{compatibility} and
the facts that $v(0)=u^2(0)=0$, we get 
\begin{equation}  \label{Th2-ker8}
u^2=-\frac{\gamma}{\varepsilon_3\xi}v.
\end{equation}
The general solution of \eqref{Th2-KER2} with $v(0)=0$, is given by 
\begin{equation}  \label{Th2-Ker9}
v(x)=B \sin\left(\lambda \sqrt{\frac{\rho}{\alpha}}x\right).
\end{equation}
Using the fact that $u^3=0$ and the boundary condition $\alpha v_x(L)+\gamma
u^3(L)=0$, we get $v_x(L)=0$. Using $v_x(L)=0$ in \eqref{Th2-Ker9}, we
obtain $B_n\la \sqrt{\frac{\rho}{\alpha}}\cos\left(\la \sqrt{\frac{\rho}{%
\alpha}}L\right)=0$. If, 
\begin{equation}  \label{Th2-Ker10}
\cos\left(\la \sqrt{\frac{\rho}{\alpha}}L\right)=0,
\end{equation}
then, 
\begin{equation}  \label{1Th2-Ker10}
\la=\frac{(2n+1)\pi}{2L}\sqrt{\frac{\alpha}{\rho}}.
\end{equation}
Using the fact that $\la^2=\frac{\mu}{\xi\varepsilon_3}$ in %
\eqref{1Th2-Ker10}, we get 
\begin{equation}  \label{Th2-Ker11}
\frac{\mu\rho}{\xi\varepsilon_3\alpha}=\frac{(2n+1)^2\pi^2}{4L^2}.
\end{equation}
This contradicts \eqref{SC}, consequently hypothesis \eqref{Th2-Ker10} is
not true, and so $v=0$, then from \eqref{Th2-KER1}, \eqref{Th2-Ker7} and %
\eqref{Th2-ker8}, we get $u^1=u^2=z=0$, which yields to $U=0$. Consequently,
if \eqref{SC} holds, then $i\la$ is not an eigenvalue of $\mathcal{A}$. Thus 
\begin{equation*}
\ker\left(i\la I-\mathcal{A}_{0,0,c}\right)=\{0\}.
\end{equation*}
On the other hand, if condition \eqref{SC} is not true (i.e;, if %
\eqref{Th2-Ker11} holds), then $i\la$ (where $\la$ is given in %
\eqref{1Th2-Ker10}) is an eigenvalue of $\mathcal{A}_{0,0,c}$ with the
corresponding eigenvector 
\begin{equation*}
U=\left(v,i\la v,i\la \frac{\gamma}{\mu}v,-\frac{\gamma}{\varepsilon_3\xi}%
v,0\right),
\end{equation*}
such that $v$ is given in \eqref{Th2-Ker9}. The proof is thus complete.
\end{proof}

%%%%%%%%%%%%%%%%%%%%%%%%%%%%%%%%%%%%%%%% 

\section{The stretching of the centreline of the beam in $x-$direction and
electrical field component in ($x$ and $z$)$-$direction are damped "$%
(a,b,c)\neq (0,0,0)$"}

\noindent The aim of this part is to prove the exponential stability of
Lorenz system \eqref{Lorenz} The stretching of the centreline of the beam in 
$x-$direction and electrical field component in ($x$ and $z$)$-$direction
are damped (i.e. $(a;b,c)\neq (0,0,0)$). The main result of this pat is the
following theorem.

\begin{theoreme}
\label{THEOREM1-EXP} The $C_0-$semigroup of contractions $(e^{t\mathcal{A}%
_{a,b,c}})_{t\geq 0}$ is exponentially stable; i.e., there exist constants $%
M\geq 1$ and $\epsilon>0$ independent of $U_0$ such that 
\begin{equation*}
\|e^{t\mathcal{A}_{a,b,c}}U_0\|_{\mathcal{H}}\leq M e^{-\epsilon t}\|U_0\|_{%
\mathcal{H}}.
\end{equation*}
\end{theoreme}

\noindent According to Huang and Pr\"uss \cite{Huang01,pruss01}, we have to
check if the following conditions hold: 
\begin{equation}  \label{Condition1-EXP}
i\mathbb{R}\subset \rho(\mathcal{A}_{a,b,c})  \tag{${\rm H1}$}
\end{equation}
and 
\begin{equation}  \label{Condition2-EXP}
\sup_{\la \in \mathbb{R}} \|(i\la I-\mathcal{A}_{a,b,c})^{-1}\|_{\mathcal{L}(%
\mathcal{H})}=O(1).  \tag{${\rm H2}$}
\end{equation}
Condition \eqref{Condition1-EXP} is already proved in Theorem \ref%
{THM-ST-ST1}. The next proposition is a technical result to be used in the
proof of \eqref{Condition2-EXP} given below.

\begin{pro}
\label{Pro-EXP1} Let $\left(\la,U:=(v,z,u^1,u^2,u^3)\right)\in \mathbb{R}%
^{\ast}\times D(\mathcal{A}_{a,b,c})$, with $\abs{\la}\geq 1$, such that 
\begin{equation}  \label{EXPO-1}
\left(i\la I-\mathcal{A}_{a,b,c}\right)U=f:=(f^1,f^2,f^3,f^4,f^5)\in 
\mathcal{H},
\end{equation}
detailed as 
\begin{eqnarray}
i\la v-z&=&f^1,  \label{EQ1-EXP1} \\
i\la z-\frac{\alpha}{\rho}v_{xx}-\frac{\gamma}{\rho}u^3_x+\frac{a}{\rho}%
z&=&f^2,  \label{EQ2-EXP1} \\
i\la u^1-u^2+u^3_x&=&f^3,  \label{EQ3-EXP1} \\
i\la u^2+\frac{\mu}{\xi \varepsilon_3}u^1+bu^2&=&f^4,  \label{EQ4-EXP1} \\
i\la u^3+\frac{\mu}{\varepsilon_3}u^1_x-\frac{\gamma}{\varepsilon_3}%
z_x+cu^3&=&f^5.  \label{EQ5-EXP1}
\end{eqnarray}
Then, we have the following inequality 
\begin{equation}  \label{EQ6}
\|U\|_{\mathcal{H}}\leq K\|F\|_{\mathcal{H}}.
\end{equation}
\end{pro}

\noindent Here and below we denote by $K_j$ a positive constant number
independent of $\la$. For the proof of Proposition \ref{Pro-EXP1}, we need
the following lemmas. %%%%%%%%%%%%%%%%%%%%%%%%%%%%%%%%%%%%%%%%%
%%%%% Lemma 1
%%%%%%%%%%%%%%%%%%%%%%%%%%%%%%%%%%%%%%%%%

\begin{lemma}
\label{Lemm1} The solution $(v,z,u^1,u^2,u^3)\in D(\mathcal{A}_{a,b,c})$ of
equation \eqref{EXPO-1}satisfies the following estimates: 
\begin{equation}  \label{Lemm1-EQ1}
\int_0^L\abs{z}^2dx\leq K_1\|U\|_{\mathcal{H}}\|F\|_{\mathcal{H}}\qquad 
\text{where}\quad K_1=\frac{1}{a},
\end{equation}
\begin{equation}  \label{Lemm1-EQ2}
\int_0^L\abs{u^2}dx\leq K_2\|U\|_{\mathcal{H}}\|F\|_{\mathcal{H}}\qquad 
\text{where}\quad K_2=\frac{1}{b\xi \varepsilon_3},
\end{equation}
\begin{equation}  \label{Lemm1-EQ3}
\int_0^L\abs{u^3}^2dx\leq K_3\|U\|_{\mathcal{H}}\|F\|_{\mathcal{H}}\qquad 
\text{where}\quad K_3=\frac{1}{c\varepsilon_3}.
\end{equation}
\end{lemma}

\begin{proof}
First, taking the inner product of \eqref{EXPO-1} with $U$ in $\mathcal{H}$,
we obtain 
\begin{equation}  \label{Lemm1-EQ4}
a\int_0^L\abs{z}^2+b\xi \varepsilon_3\int_0^L\abs{u^2}^2dx+c\varepsilon_3%
\int_0^L\abs{u^3}^2dx=\Re\left(\mathcal{A}_{a,b;c}U,U\right)\leq \|U\|_{%
\mathcal{H}}\|F\|_{\mathcal{H}}.
\end{equation}
Then, we obtain \eqref{Lemm1-EQ1}-\eqref{Lemm1-EQ3}. The proof has been
completed.
\end{proof}

%%%%%%%%%%%%%%%%%%%%%%%%%%%%%%%%%%%%%%%%%
%%%%% Lemma 2
%%%%%%%%%%%%%%%%%%%%%%%%%%%%%%%%%%%%%%%%%

\begin{lemma}
\label{Lemm2} The solution $(v,z,u^1,u^2,u^3)\in D(\mathcal{A}_{a,b,c})$ of
equation \eqref{EXPO-1}satisfies the following estimation: 
\begin{equation}  \label{Lemm2-EQ1}
\alpha\int_0^L\abs{v_x}^2dx\leq K_4\|U\|_{\mathcal{H}}\|F\|_{\mathcal{H}%
}\quad \text{where}\quad K_4=2\left(\rho K_1+2\sqrt{\frac{\rho}{\alpha}}c_p+%
\frac{ac_p^2}{\alpha}+\frac{\gamma^2}{\alpha}K_3\right).
\end{equation}
%where $K_4=\frac{2}{\alpha}\left(\rho K_1+\sqrt{\frac{\rho}{\alpha}}C_p+1+\frac{ac_p}{\alpha}+\frac{\gamma^2}{\alpha}\right)$
\end{lemma}

\begin{proof}
Multiplying \eqref{EQ2-EXP1} by $\rho \overline{v}$, integrating by parts
over $(0,L)$, we get 
\begin{equation}  \label{Lemm2-EQ2}
i\la\rho \int_0^Lz\overline{v}dx+\alpha \int_0^L\abs{v_x}^2dx+\gamma%
\int_0^Lu^3\overline{v_x}dx+a \int_0^Lz\overline{v}dx=\rho \int_0^Lf^2%
\overline{v}dx.
\end{equation}
From \eqref{EQ1-EXP1}, we get 
\begin{equation*}
\alpha \int_0^L\abs{v_x}^2dx= \rho\int_0^L\abs{z}^2dx+\rho\int_0^Lz\overline{%
f^1}dx -\gamma\int_0^Lu^3\overline{v_x}dx-a \int_0^Lz\overline{v}dx+\rho
\int_0^Lf^2\overline{v}dx,
\end{equation*}
consequently, we obtain 
\begin{equation}  \label{Lemm2-EQ3}
\alpha\int_0^L\abs{v_x}^2dx\leq \rho \int_0^L\abs{z}^2dx+\rho\int_0^L\abs{z}%
\abs{f^1}dx +\rho\int_0^L\abs{f^2}\abs{v}dx+\gamma\int_0^L\abs{u^3}\abs{v_x}%
dx+a\int_0^L\abs{z}\abs{v}dx.
\end{equation}
Using the fact that $\sqrt{\rho}\|z\|\leq\|U\|_{\mathcal{H}}$, $\sqrt{\alpha}%
\|f_x^1\|\leq \|F\|_{\mathcal{H}}$, $\sqrt{\rho}\|f_2\|\leq \|F\|_{\mathcal{H%
}}$ and Poincar\'e inequality , we get 
\begin{equation}  \label{Lemm2-EQ4}
\left\{%
\begin{array}{l}
\displaystyle \rho\int_0^L\abs{z}\abs{f^1}dx\leq \rho c_p\|z\|\|f^1_x\|\leq 
\sqrt{\frac{\rho}{\alpha}}c_p\|U\|_{\mathcal{H}}\|F\|_{\mathcal{H}}, \\ 
\displaystyle \rho\int_0^L\abs{f^2}\abs{v}dx\leq \rho c_p\|f^2\|\|v_x\|\leq 
\sqrt{\frac{\rho}{\alpha}}c_p\|U\|_{\mathcal{H}}\|F\|_{\mathcal{H}}.%
\end{array}
\right.
\end{equation}
Applying Young inequality, Poincar\'e inequality and using \eqref{Lemm1-EQ1}
and \eqref{Lemm1-EQ3}, we get 
\begin{equation}  \label{Lemm2-EQ5}
a\int_0^L\abs{z}\abs{v}dx\leq \frac{a}{2r_1}\int_0^L\abs{z}^2dx+\frac{%
ar_1c_p^2}{2}\int_0^L\abs{v_x}^2dx\leq \frac{1}{2r_1}\|U\|_{\mathcal{H}%
}\|F\|_{\mathcal{H}}+\frac{ar_1c_p^2}{2}\int_0^L\abs{v_x}^2dx
\end{equation}
and 
\begin{equation}  \label{Lemm2-EQ6}
\gamma \int_0^L\abs{u^3}\abs{v_x}dx\leq \frac{\gamma^2}{2r_2}\int_0^L%
\abs{u^3}^2dx+\frac{r_2}{2}\int_0^L\abs{v_x}^2dx\leq \frac{\gamma^2}{2r_2}%
K_3\|U\|_{\mathcal{H}}\|F\|_{\mathcal{H}}+\frac{r_2}{2}\int_0^L\abs{v_x}^2dx.
\end{equation}
Inserting \eqref{Lemm2-EQ4}-\eqref{Lemm2-EQ6} in \eqref{Lemm2-EQ3} and using %
\eqref{Lemm1-EQ1}, we get 
\begin{equation}  \label{Lemm2-EQ7}
\left(\alpha-\frac{ar_1c_p^2}{2}-\frac{r_2}{2}\right)\int_0^L\abs{v_x}%
^2dx\leq \left(\rho K_1+2\sqrt{\frac{\rho}{\alpha}}c_p+\frac{1}{2r_1}+\frac{%
\gamma^2}{2r_2}K_3\right)\|U\|_{\mathcal{H}}\|F\|_{\mathcal{H}}.
\end{equation}
Taking $r_1=\frac{\alpha}{2ac_p^2}$ and $r_2=\frac{\alpha}{2}$ in %
\eqref{Lemm2-EQ7}, we get 
\begin{equation*}
\frac{\alpha}{2}\int_0^L\abs{v_x}^2dx\leq \left(\rho K_1+2\sqrt{\frac{\rho}{%
\alpha}}c_p+\frac{ac_p^2}{\alpha}+\frac{\gamma^2}{\alpha}K_3\right)\|U\|_{%
\mathcal{H}}\|F\|_{\mathcal{H}}.
\end{equation*}
Thus, we obtain \eqref{Lemm2-EQ1}. the proof has been completed.
\end{proof}

%%%%%%%%%%%%%%%%%%%%%%%%%%%%%%%%%%%%%%%%%
%%%%% Lemma 3
%%%%%%%%%%%%%%%%%%%%%%%%%%%%%%%%%%%%%%%%%

\begin{lemma}
\label{LEMM3} The solution $(v,z,u^1,u^2,u^3)\in D(\mathcal{A}_{a,b,c})$ of
equation \eqref{EXPO-1}satisfies the following estimation: 
\begin{equation}  \label{Lemm3-EQ1}
\mu\int_0^L\abs{u^1}^2dx\leq K_5\|U\|_{\mathcal{H}}\|F\|_{\mathcal{H}},
\end{equation}
where $\displaystyle{K_5=2\xi\varepsilon_3\left(\left(1+\frac{%
b^2\xi\varepsilon_3}{2\mu}\right)K_2+\left(\frac{1}{\xi}+\frac{\gamma^2}{%
2\xi^2\varepsilon_3^2}\right)K_3+\frac{1}{2}K_4+\frac{2}{\sqrt{%
\xi\varepsilon_3\mu}}\right)}$.
\end{lemma}

\begin{proof}
Multiplying \eqref{EQ4-EXP1} by $\overline{u^1}$ integrating over $(0,L)$,
we get 
\begin{equation}  \label{Lemm3-EQ2}
i\la\int_0^Lu^2\overline{u^1}dx+\frac{\mu}{\xi\varepsilon_3}\int_0^L\abs{u^1}%
^2dx+b\int_0^Lu^2\overline{u^1}dx=\int_0^Lf^4\overline{u^1}dx.
\end{equation}
Multiplying \eqref{EQ3-EXP1} by $\overline{u^2}$ integrating by parts over $%
(0,L)$, we get 
\begin{equation}  \label{Lemm3-EQ3}
i\la\int_0^Lu^1\overline{u^2}dx-\int_0^L\abs{u^2}^2dx-\int_0^Lu^3\overline{%
u^2_x}dx=\int_0^Lf^3\overline{u^2}dx.
\end{equation}
Adding \eqref{Lemm3-EQ2}-\eqref{Lemm3-EQ3} and taking the real part, we get 
\begin{equation}  \label{Lemm3-EQ4}
\begin{array}{l}
\displaystyle \frac{\mu}{\xi\varepsilon_3}\int_0^L\abs{u^1}^2dx=\int_0^L%
\abs{u^2}^2dx-\Re\left(b\int_0^Lu^2\overline{u^1}dx\right) \\ 
\displaystyle +\Re\left(\int_0^Lu^3\overline{u^2_x}dx\right)+\Re\left(%
\int_0^Lf^4\overline{u^1}dx\right)+\Re\left(\int_0^Lf^3\overline{u^2}%
dx\right).%
\end{array}%
\end{equation}
Using the fact that $\sqrt{\xi\varepsilon_3}\|f^4\|\leq \|F\|_{\mathcal{H}}$%
, $\sqrt{\mu}\|f^3\|\leq \|F\|_{\mathcal{H}}$, $\sqrt{\xi\varepsilon_3}%
\|u^2\|\leq \|U\|_{\mathcal{H}}$ and $\sqrt{\mu}\|u^1\|\leq \|U\|_{\mathcal{H%
}}$, we get 
\begin{equation}  \label{Lemm3-EQ5}
\left\{%
\begin{array}{l}
\displaystyle \left|\Re\left(\int_0^Lf^4\overline{u^1}dx\right)\right|\leq 
\frac{1}{\sqrt{\xi\varepsilon_3\mu}}\|U\|_{\mathcal{H}}\|F\|_{\mathcal{H}},
\\[0.1in] 
\displaystyle \left|\Re\left(\int_0^Lf^3\overline{u^2}dx\right)\right|\leq 
\frac{1}{\sqrt{\xi\varepsilon_3\mu}}\|U\|_{\mathcal{H}}\|F\|_{\mathcal{H}}.%
\end{array}
\right.
\end{equation}
Using Young inequality and \eqref{Lemm1-EQ2}, we get 
\begin{equation}  \label{Lemm3-EQ6}
\left|\Re\left(b\int_0^Lu^2\overline{u^1}dx\right)\right|\leq \frac{b^2}{2r_3%
}\int_0^L\abs{u^2}^2dx+\frac{r_3}{2}\int_0^L\abs{u^1}^2dx\leq \frac{b^2}{2r_3%
}K_2\|U\|_{\mathcal{H}}\|F\|_{\mathcal{H}}+\frac{r_3}{2}\int_0^L\abs{u^1}%
^2dx.
\end{equation}
Now, we give an estimation on $\displaystyle{\Re\left(\int_0^Lu^3\overline{%
u^2_x}dx\right)}$. Using compatibility condition \eqref{compatibility}, we
get 
\begin{equation*}
\Re\left(\int_0^Lu^3\overline{u^2_x}dx\right)=\frac{1}{\xi}\int_0^L\abs{u^3}%
^2dx-\frac{\gamma}{\xi\varepsilon_3}\int_0^Lu^3\overline{v_x}dx.
\end{equation*}
Applying Young inequality in the above estimation and using \eqref{Lemm1-EQ3}
and \eqref{Lemm2-EQ1}, we get 
\begin{equation}  \label{Lemm3-EQ7}
\begin{array}{lll}
\displaystyle \left|\Re\left(\int_0^Lu^3\overline{u^2_x}dx\right)\right| & 
\leq & \displaystyle \left(\frac{1}{\xi}+\frac{\gamma^2}{2\xi^2%
\varepsilon_3^2}\right)\int_0^L\abs{u^3}^2dx+\frac{1}{2}\int_0^L\abs{v_x}^2dx
\\[0.1in] 
& \leq & \displaystyle \left(\left(\frac{1}{\xi}+\frac{\gamma^2}{%
2\xi^2\varepsilon_3^2}\right)K_3+\frac{1}{2}K_4\right)\|U\|_{\mathcal{H}%
}\|F\|_{\mathcal{H}}.%
\end{array}%
\end{equation}
Inserting \eqref{Lemm3-EQ5}, \eqref{Lemm3-EQ6}, \eqref{Lemm3-EQ7} and using %
\eqref{Lemm1-EQ2} in \eqref{Lemm3-EQ4}, we get 
\begin{equation*}
\left(\frac{\mu}{\xi\varepsilon_3}-\frac{r_3}{2}\right)\int_0^L\abs{u^1}%
^2dx\leq \left(K_2+\frac{b^2}{2r_3}K_2+\left(\frac{1}{\xi}+\frac{\gamma^2}{%
2\xi^2\varepsilon_3^2}\right)K_3+\frac{1}{2}K_4+\frac{2}{\sqrt{%
\xi\varepsilon_3\mu}}\right)\|U\|_{\mathcal{H}}\|F\|_{\mathcal{H}}.
\end{equation*}
Taking $r_3=\frac{\mu}{\xi\varepsilon_3}$ in the above estimation, we get 
\begin{equation*}
\mu\int_0^L\abs{u^1}^2dx\leq 2\xi\varepsilon_3\left(\left(1+\frac{%
b^2\xi\varepsilon_3}{2\mu}\right)K_2+\left(\frac{1}{\xi}+\frac{\gamma^2}{%
2\xi^2\varepsilon_3^2}\right)K_3+\frac{1}{2}K_4+\frac{2}{\sqrt{%
\xi\varepsilon_3\mu}}\right)\|U\|_{\mathcal{H}}\|F\|_{\mathcal{H}}.
\end{equation*}
The proof has been completed.
\end{proof}

%%%%%%%%%%%%%%%%%%%%%%%%%%%% 

\noindent \textbf{Proof of Proposition \ref{Pro-EXP1}}. Adding %
\eqref{Lemm1-EQ1}, \eqref{Lemm1-EQ2}, \eqref{Lemm1-EQ3}, \eqref{Lemm2-EQ1}
and \eqref{Lemm3-EQ1}, we get 
\begin{equation*}
\|U\|^2_{\mathcal{H}}=\alpha\|v_x\|^2+\rho\|z\|^2+\mu\|u^1\|^2+\xi%
\varepsilon_3\|u^2\|^2+\varepsilon_3\|u^3\|^2\leq K\|U\|_{\mathcal{H}%
}\|F\|_{H}.
\end{equation*}
Thus, we obtain \eqref{EQ6} where $K=\rho
K_1+\xi\varepsilon_3K_2+\varepsilon_3K_3+K_4+K_5$. 
%%%%%%%%%%%%%%%%%%%%%%%%%%%%

\noindent \textbf{Proof of Theorem \ref{THEOREM1-EXP}} For all $U\in D(%
\mathcal{A})$ according to Proposition \eqref{Pro-EXP1}, we get 
\begin{equation*}
\|U\|_{\mathcal{H}}\leq K\|(i\la I-\mathcal{A}_{a,b,c})U\|_{\mathcal{H}}.
\end{equation*}
Thus, we have 
\begin{equation*}
\|(i\la I-\mathcal{A}_{a,b,c})^{-1}V\|_{\mathcal{H}}\leq K\|V\|_{\mathcal{H}%
},\quad \forall V\in \mathcal{H}.
\end{equation*}
Therefore, from the above equation, we get \eqref{Condition2-EXP} holds.
Thus, we get the conclusion by applying Huang and Pr\"uss Theorem. 
%%%%%%%%%%%%%%%%%%%%%%%%%%%%%%%%%%%%%%%%%%%%%%%%%%%%%%%%%
%%%%%%%%%%%%%%%%%%%%%%%%%%%%%%%%%%%%%%%%%%%%%%%%%%%%%%%%%
%%%%%%%%%%%%%%%%%%%%%%%%%%%%%%%%%%%%%%%%%%%%%%%%%%%%%%%%%

\section{The electrical field component in ($x$ and $z$)$-$direction are
damped "$a=0\ \text{and}\ (b,c)\neq (0,0)$"}

\noindent The aim of this part is to prove the exponential stability of
Lorenz system \eqref{Lorenz} with the damping acting on the electrical field
component in $(x-z)-$direction i.e. ($a=0\ \text{and}\ (b,c)\neq (0,0)$).
The main result of this part is the following theorem.

\begin{theoreme}
\label{4-Thm1} The $C_0-$semigroup of contractions $(e^{t\mathcal{A}%
_{0,b,c}})_{t\geq 0}$ is exponentially stable; i.e., there exist constants $%
M\geq 1$ and $\epsilon>0$ independent of $U_0$ such that 
\begin{equation*}
\|e^{t\mathcal{A}_{0,b,c}}U_0\|_{\mathcal{H}}\leq M e^{-\epsilon t}\|U_0\|_{%
\mathcal{H}}.
\end{equation*}
\end{theoreme}

%%%%%%%%%%%%%%%%%%%%%%%
\noindent From Theorem \ref{THM-ST-ST1}, we have seen that $i\mathbb{R}%
\subset \rho(\mathcal{A}_{0,b,c})$, then for the proof of Theorem \ref%
{4-Thm1}, we still to prove that 
\begin{equation}  \label{H3}
\sup_{\la \in \mathbb{R}}\|(i\la I-\mathcal{A}_{0,b,c})^{-1}\|_{\mathcal{L}(%
\mathcal{H})}=O(1).  \tag{$\rm H3$}
\end{equation}
The next proposition is a technical result to be used in the proof of
Theorem \ref{4-Thm1} given below.

\begin{pro}
\label{PRO2} Let $\left(\la,U:=(v,z,u^1,u^2,u^3)\in \mathbb{R}^{\ast}\times
D(\mathcal{A}_{0,b,c})\right)$, with $\abs{\la}\geq 1$, such that 
\begin{equation}  \label{EQ0-pro2}
\left(i\la I-\mathcal{A}_{0,b,c}\right)U=F:=(f^1,f^2,f^3,f^4,f^5)\in 
\mathcal{H},
\end{equation}
detailed as 
\begin{eqnarray}
i\la v-z&=&f^1,  \label{EQ1-PRO2} \\
i\la z-\frac{\alpha}{\rho}v_{xx}-\frac{\gamma}{\rho}u^3_x&=&f^2,
\label{EQ2-PRO2} \\
i\la u^1-u^2+u^3_x&=&f^3,  \label{EQ3-PRO2} \\
i\la u^2+\frac{\mu}{\xi \varepsilon_3}u^1+bu^2&=&f^4,  \label{EQ4-PRO2} \\
i\la u^3+\frac{\mu}{\varepsilon_3}u^1_x-\frac{\gamma}{\varepsilon_3}%
z_x+cu^3&=&f^5.  \label{EQ5-PRO2}
\end{eqnarray}
Then, we have the following inequality 
\begin{equation}  \label{INEQ-PRO2}
\|U\|_{\mathcal{H}}\leq \mathcal{M}\|F\|_{\mathcal{H}}.
\end{equation}
\end{pro}

\noindent For the proof of Proposition \ref{PRO2}, we need the following
lemmas. %%%%%%%%%%%%%%%%%%%%%%%
%%%%%%%%%%%%%%%%%%%%%%%

\begin{lemma}
\label{CASE2-Lemm1} The solution $(v,z,u^1,u^2,u^3)\in D(\mathcal{A}%
_{0,b,c}) $ of equation \eqref{EQ0-pro2}satisfies the following estimations: 
\begin{equation}  \label{CASE2-Lemm1-EQ1}
\int_0^L\abs{u^2}dx\leq \mathcal{M}_1\|U\|_{\mathcal{H}}\|F\|_{\mathcal{H}%
}\qquad \text{where}\quad M_1=\frac{1}{b\xi \varepsilon_3},
\end{equation}
\begin{equation}  \label{CASE2-Lemm1-EQ2}
\int_0^L\abs{u^3}^2dx\leq \mathcal{M}_2\|U\|_{\mathcal{H}}\|F\|_{\mathcal{H}%
}\qquad \text{where}\quad M_2=\frac{1}{c\varepsilon_3},
\end{equation}
\end{lemma}

\begin{proof}
By using the argument in Lemma \ref{Lemm1}, we get \eqref{CASE2-Lemm1-EQ1}-%
\eqref{CASE2-Lemm1-EQ2}. The proof has been completed.
\end{proof}

%%%%%%%%%%%%%%%%%%%%%%%
%%%%%%%%%%%%%%%%%%%%%%%

\begin{lemma}
\label{CASE2-Lemm2} The solution $(v,z,u^{1},u^{2},u^{3})\in D(\mathcal{A}%
_{0,b,c})$ of equation \eqref{EQ0-pro2}satisfies the following estimation: 
\begin{equation}
\alpha \int_{0}^{L}\abs{v_x}^{2}dx\leq \mathcal{M}_{3}\Vert U\Vert _{%
\mathcal{H}}\Vert F\Vert _{\mathcal{H}},  \label{CASE2-Lemm2-EQ1}
\end{equation}%
where $\mathcal{M}_{3}=\frac{2\varepsilon _{3}\alpha }{b\gamma }\left( \frac{%
(b-c)^{2}\varepsilon _{3}}{2\gamma }M_{2}\right) $.
\end{lemma}

\begin{proof}
First, inserting \eqref{EQ1-PRO2} in \eqref{EQ5-PRO2}, we get 
\begin{equation}
i\la u^{3}+\frac{\mu }{\varepsilon _{3}}u_{x}^{1}-\frac{\gamma }{\varepsilon
_{3}}i\la v_{x}+\frac{\gamma }{\varepsilon _{3}}f_{x}^{1}+cu^{3}=f^{5}.
\label{CASE2-Lemm2-EQ2}
\end{equation}%
Multiplying \eqref{CASE2-Lemm2-EQ2} by $i\la^{-1}\overline{v_{x}}$,
integrating over $(0,L)$, we get 
\begin{equation*}
-\int_{0}^{L}u^{3}\overline{v_{x}}dx+i\frac{\mu }{\varepsilon _{3}}\la%
^{-1}\int_{0}^{L}u_{x}^{1}\overline{v_{x}}dx+\frac{\gamma }{\varepsilon _{3}}%
\int_{0}^{L}\abs{v_x}^{2}dx+i\la^{-1}\frac{\gamma }{\varepsilon _{3}}%
\int_{0}^{L}f_{x}^{1}\overline{v_{x}}dx+i\la^{-1}c\int_{0}^{L}u^{3}\overline{%
v_{x}}dx=i\la^{-1}\int_{0}^{L}f^{5}\overline{v_{x}}dx.
\end{equation*}%
Using the compatibility condition \eqref{compatibility} in the first
integral in \eqref{CASE2-Lemm2-EQ3}, we get 
\begin{equation}
-\xi \int_{0}^{L}u_{x}^{2}\overline{v_{x}}dx+i\frac{\mu }{\varepsilon _{3}}%
\la^{-1}\int_{0}^{L}u_{x}^{1}\overline{v_{x}}dx+i\la^{-1}\frac{\gamma }{%
\varepsilon _{3}}\int_{0}^{L}f_{x}^{1}\overline{v_{x}}dx+i\la%
^{-1}c\int_{0}^{L}u^{3}\overline{v_{x}}dx=i\la^{-1}\int_{0}^{L}f^{5}%
\overline{v_{x}}dx.  \label{CASE2-Lemm2-EQ3}
\end{equation}%
Deriving \eqref{EQ4-PRO2} with respect to $x$ and multiplying the result by $%
-i\xi \la^{-1}\overline{v_{x}}$, we get 
\begin{equation}
\xi \int_{0}^{L}u_{x}^{2}\overline{v_{x}}dx-i\frac{\mu }{\varepsilon _{3}}\la%
^{-1}\int_{0}^{L}u_{x}^{1}\overline{v_{x}}dx-ib\la^{-1}\xi
\int_{0}^{L}u_{x}^{2}\overline{v_{x}}dx=-i\la^{-1}\xi \int_{0}^{L}f_{x}^{4}%
\overline{v_{x}}dx.  \label{CASE2-Lemm2-EQ4}
\end{equation}%
Adding \eqref{CASE2-Lemm2-EQ3} and \eqref{CASE2-Lemm2-EQ4}, we get 
\begin{equation*}
-b\xi \int_{0}^{L}u_{x}^{2}\overline{v_{x}}dx+\frac{\gamma }{\varepsilon _{3}%
}\int_{0}^{L}f_{x}^{1}\overline{v_{x}}dx+c\int_{0}^{L}u^{3}\overline{v_{x}}%
dx=\int_{0}^{L}f^{5}\overline{v_{x}}dx-\xi \int_{0}^{L}f_{x}^{4}\overline{%
v_{x}}dx.
\end{equation*}%
Again, using the compatibility condition \eqref{compatibility} in the above
equation, we get 
\begin{equation}
\frac{\gamma }{\varepsilon _{3}}b\int_{0}^{L}\abs{v_x}^{2}dx=(b+c)%
\int_{0}^{L}u^{3}\overline{v_{x}}dx+\frac{\gamma }{\varepsilon _{3}}%
\int_{0}^{L}f_{x}^{1}\overline{v_{x}}dx-\int_{0}^{L}f^{5}\overline{v_{x}}%
dx+\xi \int_{0}^{L}f_{x}^{4}\overline{v_{x}}dx.  \label{1CASE2-Lemm2-EQ6}
\end{equation}%
Since $F\in \mathcal{H}$, then $(f^{1},f^{2},f^{3},f^{4},f^{5})$ satisfies
the compatibility condition 
\begin{equation}
\xi f_{x}^{4}-f^{5}+\frac{\gamma }{\varepsilon _{3}}f_{x}^{1}=0.
\label{COMPA-F}
\end{equation}%
Combining \eqref{COMPA-F} and \eqref{1CASE2-Lemm2-EQ6}, we get 
\begin{equation*}
\frac{\gamma }{\varepsilon _{3}}b\int_{0}^{L}\abs{v_x}^{2}dx=(b-c)%
\int_{0}^{L}u^{3}\overline{v_{x}}dx.
\end{equation*}%
It follow that 
\begin{equation}
\frac{\gamma }{\varepsilon _{3}}b\int_{0}^{L}\abs{v_x}^{2}dx\leq \abs{b-c}%
\int_{0}^{L}\abs{u^3}\abs{v_x}dx+\frac{2\gamma }{\varepsilon _{3}}%
\int_{0}^{L}\abs{f^1_x}\abs{v_x}dx+2\int_{0}^{L}\abs{f^5}\abs{v_x}dx.
\label{CASE2-Lemm2-EQ6}
\end{equation}%
Applying Young Inequality 
\begin{equation}
\begin{array}{lll}
\displaystyle\abs{b-c}\int_{0}^{L}\abs{u^3}\abs{v_x}dx & \leq & \displaystyle%
\frac{(b-c)^{2}\varepsilon _{3}}{2\gamma }\int_{0}^{L}\abs{u^3}^{2}+\frac{%
b\gamma }{2\varepsilon _{3}}\int_{0}^{L}\abs{v_x}^{2}dx, \\[0.1in] 
& \leq & \displaystyle\frac{(b-c)^{2}\varepsilon _{3}}{2\gamma }M_{2}\Vert
U\Vert _{\mathcal{H}}\Vert F\Vert _{\mathcal{H}}+\frac{b\gamma }{%
2\varepsilon _{3}}\int_{0}^{L}\abs{v_x}^{2}dx%
\end{array}
\label{CASE2-Lemm2-EQ8}
\end{equation}%
Inserting \eqref{CASE2-Lemm2-EQ8} in \eqref{CASE2-Lemm2-EQ6}, we get 
\begin{equation*}
\frac{\gamma }{2\varepsilon _{3}}b\int_{0}^{L}\abs{v_x}^{2}dx\leq \left( 
\frac{(b-c)^{2}\varepsilon _{3}}{2\gamma }M_{2}\right) \Vert U\Vert _{%
\mathcal{H}}\Vert F\Vert _{\mathcal{H}}.
\end{equation*}%
Thus, we obtain \eqref{CASE2-Lemm2-EQ1}. The proof has been completed.
\end{proof}

%%%%%%%%%%%%%%%%%%%%%%%%%%%%%%%%%%%%%%%%%%

\begin{lemma}
\label{CASE2-Lemm3} The solution $(v,z,u^1,u^2,u^3)\in D(\mathcal{A}%
_{0,b,c}) $ of equation \eqref{EQ0-pro2}satisfies the following estimation: 
\begin{equation}  \label{CASE2-Lemm3-EQ1}
\int_0^L\abs{z}^2dx\leq \mathcal{M}_4\|U\|_{\mathcal{H}}\|F\|_{\mathcal{H}},
\end{equation}
where $\mathcal{M}_4=2+\frac{\gamma}{\sqrt{\varepsilon_3\alpha}}+\frac{\gamma%
}{2}M_2+\left(1+\frac{\gamma}{2}\right)\mathcal{M}$.
\end{lemma}

\begin{proof}
Multiplying \eqref{EQ2-PRO2} by $-i\la^{-1}\rho\overline{z}$, integrating by
parts over $(0,L)$, we get 
\begin{equation}  \label{CASE2-Lemm3-EQ2}
\rho \int_0^L\abs{z}^2dx-i\la^{-1}\alpha\int_0^Lv_x\overline{z_x}%
dx-i\gamma\lambda^{-1}\int_0^Lu^3\overline{z_x}dx=-i\la^{-1}\rho\int_0^Lf^2%
\overline{z}dx.
\end{equation}
From \eqref{EQ1-PRO2}, we have 
\begin{equation}  \label{CASE2-Lemm3-EQ3}
-i\la^{-1}\overline{z_x}=-\overline{v_x}+i\la^{-1}\overline{f^1_x}
\end{equation}
Inserting \eqref{CASE2-Lemm3-EQ3} in \eqref{CASE2-Lemm3-EQ2}, we get 
\begin{equation*}
\rho \int_0^L\abs{z}^2dx= \alpha\int_0^L\abs{v_x}^2dx-i\la%
^{-1}\alpha\int_0^Lv_x\overline{f_x^1}dx+\gamma \int_0^Lu^3\overline{v_x}%
dx-i\gamma\la^{-1}\int_0^Lu^3\overline{f^1_x}dx-i\la^{-1}\rho\int_0^Lf^2%
\overline{z}dx.
\end{equation*}
Consequently, we get 
\begin{equation}  \label{CASE2-Lemm3-EQ4}
\begin{array}{l}
\displaystyle \rho \int_0^L\abs{z}^2dx\leq \alpha\int_0^L\abs{v_x}^2dx+%
\abs{\la}^{-1}\alpha\int_0^L\abs{v_x}\abs{f_x^1}dx \\[0.1in] 
\displaystyle +\gamma\int_0^L\abs{u^3}\abs{v_x}dx+\gamma\abs{\la}%
^{-1}\int_0^L\abs{u^3}\abs{f_x^1}dx+\rho\abs{\la}^{-1}\int_0^L\abs{f^2}%
\abs{z}dx.%
\end{array}%
\end{equation}
Using the fact that $\sqrt{\alpha}\|v_x\|\leq \|U\|_{\mathcal{H}}$, $\sqrt{%
\rho}\|z\|\leq \|F\|_{\mathcal{H}}$, $\sqrt{\varepsilon_3}\|u^3\|\leq \|U\|_{%
\mathcal{H}}$, $\sqrt{\alpha}\|f^1_x\|\leq \|F\|_{\mathcal{H}}$ and $%
\abs{\la}\geq 1$, we get 
\begin{equation}  \label{CASE2-Lemm3-EQ5}
\left\{%
\begin{array}{l}
\displaystyle \abs{\la}^{-1}\alpha\int_0^L\abs{v_x}\abs{f_x^1}dx\leq \|U\|_{%
\mathcal{H}}\|F\|_{\mathcal{H}}, \\[0.1in] 
\displaystyle \gamma\abs{\la}^{-1}\int_0^L\abs{u^3}\abs{f_x^1}dx\leq \frac{%
\gamma}{\sqrt{\varepsilon_3\alpha}}\|U\|_{\mathcal{H}}\|F\|_{\mathcal{H}}, \\%
[0.1in] 
\displaystyle \rho\abs{\la}^{-1}\int_0^L\abs{f^2}\abs{z}dx\leq \|U\|_{%
\mathcal{H}}\|F\|_{\mathcal{H}}.%
\end{array}
\right.
\end{equation}
Applying Young inequality and using \eqref{CASE2-Lemm1-EQ2} and %
\eqref{CASE2-Lemm2-EQ1}, we get 
\begin{equation}  \label{CASE2-Lemm3-EQ6}
\gamma\int_0^L\abs{u^3}\abs{v_x}dx\leq \frac{\gamma}{2}\int_0^L\abs{u^3}^2dx+%
\frac{\gamma}{2}\int_0^L\abs{v_x}^2dx\leq \frac{\gamma}{2}\left(\mathcal{M}%
_2+\frac{1}{\alpha}\mathcal{M}_3\right)\|U\|_{\mathcal{H}}\|F\|_{\mathcal{H}%
}.
\end{equation}
Inserting \eqref{CASE2-Lemm3-EQ5} and \eqref{CASE2-Lemm3-EQ6} in %
\eqref{CASE2-Lemm3-EQ4} and using \eqref{CASE2-Lemm2-EQ1}, we get 
\begin{equation*}
\rho \int_0^L\abs{z}^2dx\leq \left(2+\frac{\gamma}{\sqrt{\varepsilon_3\alpha}%
}+\frac{\gamma}{2}\mathcal{M}_2+\left(1+\frac{\gamma}{2\alpha}\right)%
\mathcal{M}_3\right)\|U\|_{\mathcal{H}}\|F\|_{\mathcal{H}}.
\end{equation*}
The proof has been completed.
\end{proof}

%%%%%%%%%%%%%%%%%%%%%%%%%%%%%%%%%%%%%%%%%%

\begin{lemma}
\label{Case2-Lemm4} The solution $(v,z,u^1,u^2,u^3)\in D(\mathcal{A}%
_{0,b,c}) $ of equation \eqref{EQ0-pro2}satisfies the following estimation: 
\begin{equation}  \label{CASE2-Lemm4-EQ1}
\mu\int_0^L\abs{u^1}^2dx\leq \mathcal{M}_5\|U\|_{\mathcal{H}}\|F\|_{\mathcal{%
H}},
\end{equation}
where $\mathcal{M}_5=\xi\varepsilon_3\left(\left(1+\frac{b^2\xi\varepsilon_3%
}{2\mu}\right)\mathcal{M}_1+\left(\frac{1}{\xi}+\frac{\gamma^2}{%
2\xi^2\varepsilon_3^2}\right)\mathcal{M}_2+\frac{1}{2}\mathcal{M}_3+\frac{2}{%
\sqrt{\xi\varepsilon_3\mu}}\right)$.
\end{lemma}

\begin{proof}
By proceeding the same technics used in Lemma \ref{LEMM3}, we get %
\eqref{CASE2-Lemm4-EQ1}. The proof has been completed.
\end{proof}

%%%%%%%%%%%%%%%%%%%%%%%%%%%%%%%%%%%%%%%%%%
\noindent \textbf{Proof of Proposition \ref{PRO2}}. Adding %
\eqref{CASE2-Lemm1-EQ1}, \eqref{CASE2-Lemm1-EQ2}, \eqref{CASE2-Lemm2-EQ1}, %
\eqref{CASE2-Lemm3-EQ1} and \eqref{CASE2-Lemm4-EQ1}, we get 
\begin{equation*}
\|U\|^2_{\mathcal{H}}=\alpha\|v_x\|^2+\rho\|z\|^2+\mu\|u^1\|^2+\xi%
\varepsilon_3\|u^2\|^2+\varepsilon_3\|u^3\|^2\leq M\|U\|_{\mathcal{H}%
}\|F\|_{H},
\end{equation*}
where $M=\mathcal{M}_3+\mathcal{M}_4+\mathcal{M}_5+\xi\varepsilon_3 \mathcal{%
M}_1+\varepsilon \mathcal{M}_2$. Then, $\|U\|_{\mathcal{H}}\leq \mathcal{M}%
\|F\|_{\mathcal{H}}$. The proof has been completed.\\[0.1in]
%%%%%%%%%%%%%%%%%%%%%%%%%%%%%%%%%%%%%%%%%%
\noindent \textbf{Proof of Theorem \ref{4-Thm1}} For all $U\in D(\mathcal{A}%
_{0,b,c})$ according to Proposition \eqref{PRO2}, we get 
\begin{equation*}
\|U\|_{\mathcal{H}}\leq \mathcal{M}\|(i\la I-\mathcal{A}_{0,b,c})U\|_{%
\mathcal{H}}.
\end{equation*}
Thus, we have 
\begin{equation*}
\|(i\la I-\mathcal{A}_{0,b,c})^{-1}V\|_{\mathcal{H}}\leq \mathcal{M}\|V\|_{%
\mathcal{H}},\quad \forall V\in \mathcal{H}.
\end{equation*}
Therefore, from the above equation, we get \eqref{H3} holds. Thus, we get
the conclusion by applying Huang and Pr\"uss Theorem. 
%%%%%%%%%%%%%%%%%%%%%%%%%%%%%%%%%%%%%%%%%%%%%%%%%%%%%%%%%
%%%%%%%%%%%%%%%%%%%%%%%%%%%%%%%%%%%%%%%%%%%%%%%%%%%%%%%%%
%%%%%%%%%%%%%%%%%%%%%%%%%%%%%%%%%%%%%%%%%%%%%%%%%%%%%%%%%

\section{The stretching of the centreline of the beam in $x-$direction and
electrical field component in $z-$direction are damped "$b=0\ \text{and}\
(a,c)\neq (0,0)$"}

\noindent The aim of this part is to prove the exponential stability of
Lorenz system \eqref{Lorenz} with a damping acting on the stretching of the
centerline of the beam in $x-$direction and electrical field component in $%
x- $direction. (i.e. $b=0\quad \text{and}\quad(a,c)\neq (0,0)$). The main
result of this pat is the following theorem.

\begin{theoreme}
\label{THEOREM1-EXP-Case3} The $C_0-$semigroup of contractions $(e^{t%
\mathcal{A}_{a,0,c}})_{t\geq 0}$ is exponentially stable; i.e., there exist
constants $M\geq 1$ and $\epsilon>0$ independent of $U_0$ such that 
\begin{equation*}
\|e^{t\mathcal{A}_{a,0,c}}U_0\|_{\mathcal{H}}\leq M e^{-\epsilon t}\|U_0\|_{%
\mathcal{H}}.
\end{equation*}
\end{theoreme}

\noindent According to Huang and Pr\"uss, we have to check if the following
conditions hold: 
\begin{equation}  \label{Condition1-EXP}
i\mathbb{R}\subset \rho(\mathcal{A}_{a,0,c})  \tag{${\rm H1}$}
\end{equation}
and 
\begin{equation}  \label{CASE3-Condition2-EXP}
\sup_{\la \in \mathbb{R}} \|(i\la I-\mathcal{A}_{a,0,c})^{-1}\|_{\mathcal{L}(%
\mathcal{H})}=O(1).  \tag{${\rm H5}$}
\end{equation}
Condition \eqref{Condition1-EXP} is already proved in Theorem \ref%
{THM-ST-ST1}. The next proposition is a technical result to be used in the
proof of \eqref{CASE3-Condition2-EXP} given below.

\begin{pro}
\label{PRO3} Let $\left( \la,U:=(v,z,u^{1},u^{2},u^{3})\right) \in \mathbb{R}%
^{\ast }\times D(\mathcal{A}_{a,0,c})$, with $\abs{\la}\geq 1$, such that 
\begin{equation}
\left( i\la I-\mathcal{A}_{a,0,c}\right)
U=f:=(f^{1},f^{2},f^{3},f^{4},f^{5})\in \mathcal{H},  \label{CASE3-EXPO-1}
\end{equation}%
detailed as 
\begin{eqnarray}
i\la v-z &=&f^{1},  \label{CASE3-EQ1-EXP1} \\
i\la z-\frac{\alpha }{\rho }v_{xx}-\frac{\gamma }{\rho }u_{x}^{3}+\frac{a}{%
\rho }z &=&f^{2},  \label{CASE3-EQ2-EXP1} \\
i\la u^{1}-u^{2}+u_{x}^{3} &=&f^{3},  \label{CASE3-EQ3-EXP1} \\
i\la u^{2}+\frac{\mu }{\xi \varepsilon _{3}}u^{1} &=&f^{4},
\label{CASE3-EQ4-EXP1} \\
i\la u^{3}+\frac{\mu }{\varepsilon _{3}}u_{x}^{1}-\frac{\gamma }{\varepsilon
_{3}}z_{x}+cu^{3} &=&f^{5}.  \label{CASE3-EQ5-EXP1}
\end{eqnarray}%
Then, we have the following inequality 
\begin{equation}
\Vert U\Vert _{\mathcal{H}}\leq \mathcal{N}\Vert F\Vert _{\mathcal{H}}.
\label{CASE3-EQ6}
\end{equation}
\end{pro}

%%%%%%%%%%%%%%%%%%%%%%%
\noindent For the proof of Proposition \ref{PRO3}, we need the following
lemmas. %%%%%%%%%%%%%%%%%%%%%%%
%%%%%%%%%%%%%%%%%%%%%%%

\begin{lemma}
\label{Case3-Lemm1} The solution $(v,z,u^1,u^2,u^3)\in D(\mathcal{A}%
_{a,0,c}) $ of equation \eqref{CASE3-EXPO-1} satisfies the following
estimates: 
\begin{equation}  \label{CASE3-Lemm1-EQ1}
\int_0^L\abs{z}^2dx\leq \mathcal{N}_1\|U\|_{\mathcal{H}}\|F\|_{\mathcal{H}%
}\qquad \text{where}\quad \mathcal{N}_1=\frac{1}{a},
\end{equation}
\begin{equation}  \label{CASE3-Lemm1-EQ2}
\int_0^L\abs{u^3}dx\leq \mathcal{N}_2\|U\|_{\mathcal{H}}\|F\|_{\mathcal{H}%
}\qquad \text{where}\quad \mathcal{N}_2=\frac{1}{c\varepsilon_3},
\end{equation}
\end{lemma}

\begin{proof}
By using the same argument used in Lemma \ref{Lemm1}, we get %
\eqref{CASE3-Lemm1-EQ1}-\eqref{CASE3-Lemm1-EQ2}. The proof has been
completed.
\end{proof}

%%%%%%%%%%%%%%%%%%%%%%%
%%%%%%%%%%%%%%%%%%%%%%%

\begin{lemma}
\label{Case3-Lemm2} The solution $(v,z,u^1,u^2,u^3)\in D(\mathcal{A}%
_{a,0,c}) $ of equation \eqref{CASE3-EXPO-1} satisfies the following
estimation: 
\begin{equation}  \label{Case3-Lemm2-EQ1}
\alpha\int_0^L\abs{v_x}^2dx\leq \mathcal{N}_3\|U\|_{\mathcal{H}}\|F\|_{%
\mathcal{H}}\quad \text{where}\quad \mathcal{N}_3=2\left(\rho \mathcal{N}_1+2%
\sqrt{\frac{\rho}{\alpha}}c_p+\frac{ac_p^2}{\alpha}+\frac{\gamma^2}{\alpha}%
\mathcal{N}_2\right).
\end{equation}
\end{lemma}

\begin{proof}
Using the same arguments in Lemma \ref{Lemm2}, we get \eqref{Case3-Lemm2-EQ1}%
. The proof has been completed.
\end{proof}

%%%%%%%%%%%%%%%%%%%%%%%
%%%%%%%%%%%%%%%%%%%%%%%

\begin{lemma}
\label{Case3-Lemm3} The solution $(v,z,u^1,u^2,u^3)\in D(\mathcal{A}%
_{a,0,c}) $ of equation \eqref{CASE3-EXPO-1} satisfies the following
estimation: 
\begin{equation}  \label{Case3-Lemm3-EQ2}
\xi\varepsilon_3\int_0^L\abs{u^2}^2dx\leq \mathcal{N}_{4}\|U\|_{\mathcal{H}%
}\|F\|_{\mathcal{H}},\quad \text{where}\quad \mathcal{N}_4=2\left(\frac{%
\varepsilon_3}{\xi}\mathcal{N}_2+\frac{\gamma^2}{\varepsilon_3\xi\alpha}%
\mathcal{N}_3\right).
\end{equation}
\end{lemma}

\begin{proof}
Multiplying the compatibility condition \eqref{compatibility} by $%
\varepsilon _{3}\overline{u_{x}^{2}}$, integrating over $(0,L)$, we get 
\begin{equation*}
\xi \varepsilon _{3}\int_{0}^{L}\abs{u^2_x}^{2}dx=\varepsilon
_{3}\int_{0}^{L}u^{3}\overline{u_{x}^{2}}dx-\gamma \int_{0}^{L}v_{x}%
\overline{u_{x}^{2}}dx,
\end{equation*}%
it yields that 
\begin{equation}
\xi \varepsilon _{3}\int_{0}^{L}\abs{u^2_x}^{2}dx\leq \varepsilon
_{3}\int_{0}^{L}\abs{u^3}\abs{u^2_x}dx+\gamma \int_{0}^{L}\abs{v_x}%
\abs{u^2_x}dx.  \label{Case3-Lemm3-EQ3}
\end{equation}%
Applying Young inequality in \eqref{Case3-Lemm3-EQ3}, we get 
\begin{equation*}
\xi \varepsilon _{3}\int_{0}^{L}\abs{u^2_x}^{2}dx\leq \frac{\varepsilon _{3}%
}{\xi r}\int_{0}^{L}\abs{u^3}^{2}dx+\frac{\gamma ^{2}}{\varepsilon _{3}\xi r}%
\int_{0}^{L}\abs{v_x}^{2}dx+\frac{r}{2}\xi \varepsilon _{3}\int_{0}^{L}%
\abs{u_x^2}^{2}dx
\end{equation*}%
By taking $r=1$ in the above estimation and using \eqref{CASE3-Lemm1-EQ2}
and \eqref{Case3-Lemm2-EQ1}, we get 
\begin{equation}
\frac{\xi \varepsilon _{3}}{2}\int_{0}^{L}\abs{u^2_x}^{2}dx\leq \frac{%
\varepsilon _{3}}{\xi }\int_{0}^{L}\abs{u^3}^{2}dx+\frac{\gamma ^{2}}{%
\varepsilon _{3}\xi }\int_{0}^{L}\abs{v_x}^{2}dx\leq \left( \frac{%
\varepsilon _{3}}{\xi }\mathcal{N}_{2}+\frac{\gamma ^{2}}{\varepsilon
_{3}\xi \alpha }\mathcal{N}_{3}\right) \Vert U\Vert _{\mathcal{H}}\Vert
F\Vert _{\mathcal{H}}.  \label{Case3-Lemm3-EQ4}
\end{equation}%
Thus, we obtain \eqref{Case3-Lemm2-EQ1}. The proof is thus completed.
\end{proof}

%%%%%%%%%%%%%%%%%%%%%%%%%%%%%%%%%%%%%%%%%%

\begin{lemma}
\label{Case3-Lemm4} The solution $(v,z,u^1,u^2,u^3)\in D(\mathcal{A}%
_{a,0,c}) $ of equation \eqref{CASE3-EXPO-1}satisfies the following
estimation: 
\begin{equation}  \label{CASE3-Lemm4-EQ1}
\mu\int_0^L\abs{u^1}^2dx\leq \mathcal{N}_5\|U\|_{\mathcal{H}}\|F\|_{\mathcal{%
H}},
\end{equation}
where $\mathcal{N}_5=2\xi\varepsilon_3\left(\left(1+\frac{b^2\xi\varepsilon_3%
}{2\mu}\right)\mathcal{N}_4+\left(\frac{1}{\xi}+\frac{\gamma^2}{%
2\xi^2\varepsilon_3^2}\right)\mathcal{N}_2+\frac{1}{2}\mathcal{N}_3+\frac{2}{%
\sqrt{\xi\varepsilon_3\mu}}\right)$.
\end{lemma}

\begin{proof}
By proceeding the same technics used in Lemma \ref{LEMM3}, we get %
\eqref{CASE3-Lemm4-EQ1}. The proof has been completed.
\end{proof}

%%%%%%%%%%%%%%%%%%%%%%%%%%%%%%%%%%%%%%%%%%
\noindent \textbf{Proof of Proposition \ref{PRO3}}. Adding %
\eqref{CASE3-Lemm1-EQ1}, \eqref{CASE3-Lemm1-EQ2}, \eqref{Case3-Lemm2-EQ1}, %
\eqref{Case3-Lemm3-EQ2} and \eqref{CASE3-Lemm4-EQ1}, we get 
\begin{equation*}
\|U\|^2_{\mathcal{H}}=\alpha\|v_x\|^2+\rho\|z\|^2+\mu\|u^1\|^2+\xi%
\varepsilon_3\|u^2\|^2+\varepsilon_3\|u^3\|^2\leq \mathcal{N}\|U\|_{\mathcal{%
H}}\|F\|_{H},
\end{equation*}
where $\mathcal{N}=\mathcal{N}_3+\rho\mathcal{N}_1+\mathcal{N}_5+\mathcal{N}%
_4+\varepsilon_3 \mathcal{N}_2$. Then, $\|U\|_{\mathcal{H}}\leq \mathcal{N}%
\|F\|_{\mathcal{H}}$. The proof has been completed.\\[0.1in]
%%%%%%%%%%%%%%%%%%%%%%%%%%%%%%%%%%%%%%%%%%
%%%%%%%%%%%%%%%%%%%%%%%%%%%%%%%%%%%%%%%%%%
\noindent \textbf{Proof of Theorem \ref{THEOREM1-EXP-Case3}} For all $U\in D(%
\mathcal{A}_{a,0,c})$ according to Proposition \eqref{Pro-EXP1}, we get 
\begin{equation*}
\|U\|_{\mathcal{H}}\leq K\|(i\la I-\mathcal{A}_{a,0,c})U\|_{\mathcal{H}}.
\end{equation*}
Thus, we have 
\begin{equation*}
\|(i\la I-\mathcal{A}_{a,0,c})^{-1}V\|_{\mathcal{H}}\leq K\|V\|_{\mathcal{H}%
},\quad \forall V\in \mathcal{H}.
\end{equation*}
Therefore, from the above equation, we get \eqref{CASE3-Condition2-EXP}
holds. Thus, we get the conclusion by applying Huang and Pr\"uss Theorem. 
%%%%%%%%%%%%%%%%%%%%%%%%%%%%%%%%%%%%%%%%%%%%%%%%%%%%%%%%%
%%%%%%%%%%%%%%%%%%%%%%%%%%%%%%%%%%%%%%%%%%%%%%%%%%%%%%%%%
%%%%%%%%%%%%%%%%%%%%%%%%%%%%%%%%%%%%%%%%%%%%%%%%%%%%%%%%%

\section{The stretching of the centreline of the beam in $x-$direction and
electrical field component in $x-$direction are damped "$c=0\ \text{and}\
(a,b)\neq (0,0)$"}

\noindent The aim of this part is to prove the exponential stability of
Lorenz system \eqref{Lorenz} with a damping acting on the stretching of the
centerline of the beam in $x-$direction and electrical field component in $%
z- $direction. (i.e. $c=0\quad \text{and}\quad(a,b)\neq (0,0)$). The main
result of this pat is the following theorem.

\begin{theoreme}
\label{THEOREM1-EXP-Case4} The $C_0-$semigroup of contractions $(e^{t%
\mathcal{A}_{a,b,0}})_{t\geq 0}$ is exponentially stable; i.e., there exist
constants $M\geq 1$ and $\epsilon>0$ independent of $U_0$ such that 
\begin{equation*}
\|e^{t\mathcal{A}_{a,b,0}}U_0\|_{\mathcal{H}}\leq M e^{-\epsilon t}\|U_0\|_{%
\mathcal{H}}.
\end{equation*}
\end{theoreme}

\noindent According to Huang and Pr\"uss, we have to check if the following
conditions hold: 
\begin{equation}  \label{Condition1-EXP}
i\mathbb{R}\subset \rho(\mathcal{A}_{a,b,0})  \tag{${\rm H1}$}
\end{equation}
and 
\begin{equation}  \label{CASE4-Condition2-EXP}
\sup_{\la \in \mathbb{R}} \|(i\la I-\mathcal{A}_{a,b,0})^{-1}\|_{\mathcal{L}(%
\mathcal{H})}=O(1).  \tag{${\rm H6}$}
\end{equation}

\noindent Condition \eqref{Condition1-EXP} is already proved in Theorem \ref%
{THM-ST-ST1}. The next proposition is a technical result to be used in the
proof of \eqref{CASE4-Condition2-EXP} given below.

\begin{pro}
\label{CASE4-PRO3} Let $\left(\la,U:=(v,z,u^1,u^2,u^3)\right)\in \mathbb{R}%
^{\ast}\times D(\mathcal{A}_{a,0,c})$, with $\abs{\la}\geq 1$, such that 
\begin{equation}  \label{CASE4-EXPO-1}
\left(i\la I-\mathcal{A}_{a,b,0}\right)U=f:=(f^1,f^2,f^3,f^4,f^5)\in 
\mathcal{H},
\end{equation}
detailed as 
\begin{eqnarray}
i\la v-z&=&f^1,  \label{CASE4-EQ1-EXP1} \\
i\la z-\frac{\alpha}{\rho}v_{xx}-\frac{\gamma}{\rho}u^3_x+\frac{a}{\rho}%
z&=&f^2,  \label{CASE4-EQ2-EXP1} \\
i\la u^1-u^2+u^3_x&=&f^3,  \label{CASE4-EQ3-EXP1} \\
i\la u^2+\frac{\mu}{\xi \varepsilon_3}u^1+bu^2&=&f^4,  \label{CASE4-EQ4-EXP1}
\\
i\la u^3+\frac{\mu}{\varepsilon_3}u^1_x-\frac{\gamma}{\varepsilon_3}%
z_x&=&f^5.  \label{CASE4-EQ5-EXP1}
\end{eqnarray}
Then, we have the following inequality 
\begin{equation}  \label{CASE4-EQ6}
\|U\|_{\mathcal{H}}\leq \mathcal{S}\|F\|_{\mathcal{H}}.
\end{equation}
\end{pro}

%%%%%%%%%%%%%%%%%%%%%%%
\noindent For the proof of Proposition \ref{CASE4-PRO3}, we need the
following lemmas. %%%%%%%%%%%%%%%%%%%%%%%

\begin{lemma}
\label{Case4-Lemm1} The solution $(v,z,u^1,u^2,u^3)\in D(\mathcal{A}%
_{a,b,0}) $ of equation \eqref{CASE4-EXPO-1} satisfies the following
estimates: 
\begin{equation}  \label{CASE4-Lemm1-EQ1}
\int_0^L\abs{z}^2dx\leq \mathcal{S}_1\|U\|_{\mathcal{H}}\|F\|_{\mathcal{H}%
}\qquad \text{where}\quad \mathcal{S}_1=\frac{1}{a},
\end{equation}
\begin{equation}  \label{CASE4-Lemm1-EQ2}
\int_0^L\abs{u^2}dx\leq \mathcal{S}_2\|U\|_{\mathcal{H}}\|F\|_{\mathcal{H}%
}\qquad \text{where}\quad \mathcal{S}_2=\frac{1}{b\xi \varepsilon_3},
\end{equation}
\end{lemma}

\begin{proof}
By using the same arguments used in Lemma \ref{Lemm1}, we get %
\eqref{CASE4-Lemm1-EQ1}-\eqref{CASE4-Lemm1-EQ2}. The proof has been
completed.
\end{proof}

%%%%%%%%%%%%%%%%%%%%%%%
%%%%%%%%%%%%%%%%%%%%%%%

\begin{lemma}
\label{CASE4-lemma2} The solution $(v,z,u^1,u^2,u^3)\in D(\mathcal{A}%
_{a,b,0})$ of equation \eqref{CASE4-EXPO-1} satisfies the following
estimates: 
\begin{equation}  \label{CASE4-lemma2-EQ1}
\frac{\alpha}{2}\int_0^L\abs{v_x}^2dx+\varepsilon_3\int_0^L\abs{u^3}^2dx\leq 
\mathcal{S}_3\|U\|_{\mathcal{H}}\|F\|_{\mathcal{H}},
\end{equation}
where $\mathcal{S}_3=2\left(\sqrt{\frac{\rho}{\alpha}}c_p+b^{-1}+\frac{%
\gamma b^{-1}}{\sqrt{\alpha\varepsilon_3}}+\frac{ac^2_p}{4\alpha}\right)+%
\mathcal{S}_1$.
\end{lemma}

\begin{proof}
The proof of this Lemma, is divided into several Steps.\newline
\textbf{Step 1.} The aim of this step is to prove the following equation 
\begin{equation}
\alpha \int_{0}^{L}\abs{v_x}^{2}dx+\gamma \int_{0}^{L}u^{3}\overline{v_{x}}%
dx=\rho \int_{0}^{L}\abs{z}^{2}dx+\rho \int_{0}^{L}z\overline{f^{1}}%
dx-a\int_{0}^{L}z\overline{v}dx+\rho \int_{0}^{L}f^{2}\overline{v}dx.
\label{CASE4-lemma2-EQ1-Step1}
\end{equation}%
For this aim, multiplying \eqref{CASE4-EQ2-EXP1} by $\rho \overline{v}$,
integrating by parts over $(0,L)$, we get 
\begin{equation*}
i\la\rho \int_{0}^{L}z\overline{v}dx+\alpha \int_{0}^{L}\abs{v_x}%
^{2}dx+\gamma \int_{0}^{L}u^{3}\overline{v}_{x}dx+a\int_{0}^{L}z\overline{v}%
dx=\rho \int_{0}^{L}f^{2}\overline{v}dx.
\end{equation*}%
Using \eqref{CASE4-EQ1-EXP1} in the above equation, we get %
\eqref{CASE4-lemma2-EQ1-Step1}.\newline
\textbf{Step 2.} The aim of this step is to prove the following equation 
\begin{equation}
\varepsilon _{3}\int_{0}^{L}\abs{u^3}^{2}dx-\gamma \int_{0}^{L}v_{x}%
\overline{u^{3}}dx=b^{-1}\xi \varepsilon _{3}\int_{0}^{L}f_{x}^{4}\overline{%
u^{3}}dx-b^{-1}\varepsilon _{3}\int_{0}^{L}f^{5}\overline{u^{3}}dx+\gamma
b^{-1}\int_{0}^{L}f_{x}^{1}\overline{u^{3}}dx.
\label{CASE4-lemma2-EQ1-Step2}
\end{equation}%
For this aim, inserting \eqref{CASE4-EQ1-EXP1} in \eqref{CASE4-EQ5-EXP1}, we
get 
\begin{equation*}
i\la u^{3}+\frac{\mu }{\varepsilon _{3}}u_{x}^{1}-\frac{\gamma }{\varepsilon
_{3}}i\la v_{x}+\frac{\gamma }{\varepsilon _{3}}f_{x}^{1}=f^{5}.
\end{equation*}%
Multiplying the above equation by $-i\la^{-1}\varepsilon _{3}\overline{u^{3}}
$, integrating over $(0,L)$, we get 
\begin{equation}
\varepsilon _{3}\int_{0}^{L}\abs{u^3}^{2}dx-i\la^{-1}\mu
\int_{0}^{L}u_{x}^{1}\overline{u^{3}}dx-\gamma \int_{0}^{L}v_{x}\overline{%
u^{3}}dx=-i\la^{-1}\varepsilon _{3}\int_{0}^{L}f^{5}\overline{u^{3}}dx+i\la%
^{-1}\gamma \int_{0}^{L}f_{x}^{1}\overline{u^{3}}dx.
\label{CASE4-lemma2-EQ2-Step2}
\end{equation}%
Differentiating \eqref{CASE4-EQ4-EXP1} with respect to $x$, we obtain 
\begin{equation*}
i\la u_{x}^{2}+\frac{\mu }{\xi \varepsilon _{3}}%
u_{x}^{1}+bu_{x}^{2}=f_{x}^{4}.
\end{equation*}%
Multiplying the above equation by $i\la^{-1}\xi \varepsilon _{3}\overline{%
u^{3}}$, integrating over $(0,L)$, we get 
\begin{equation}
-\xi \varepsilon _{3}\int_{0}^{L}u_{x}^{2}\overline{u^{3}}dx+i\la^{-1}\mu
\int_{0}^{L}u_{x}^{1}\overline{u^{3}}dx+ib\la^{-1}\xi \varepsilon
_{3}\int_{0}^{L}u_{x}^{2}\overline{u^{3}}dx=i\la^{-1}\xi \varepsilon
_{3}\int_{0}^{L}f_{x}^{4}\overline{u^{3}}dx.  \label{CASE4-lemma2-EQ3-Step2}
\end{equation}%
Using the compatibility condition \eqref{compatibility} in the first term of %
\eqref{CASE4-lemma2-EQ3-Step2}, we get 
\begin{equation}
-\varepsilon _{3}\int_{0}^{L}\abs{u^3}^{2}dx+\gamma \int_{0}^{L}v_{x}%
\overline{u^{3}}dx+i\la^{-1}\mu \int_{0}^{L}u_{x}^{1}\overline{u^{3}}dx+ib\la%
^{-1}\xi \varepsilon _{3}\int_{0}^{L}u_{x}^{2}\overline{u^{3}}dx=i\la%
^{-1}\xi \varepsilon _{3}\int_{0}^{L}f_{x}^{4}\overline{u^{3}}dx.
\label{CASE4-lemma2-EQ4-Step2}
\end{equation}%
Now, adding \eqref{CASE4-lemma2-EQ2-Step2} and \eqref{CASE4-lemma2-EQ4-Step2}%
, we get 
\begin{equation}
b\xi \varepsilon _{3}\int_{0}^{L}u_{x}^{2}\overline{u^{3}}dx=\xi \varepsilon
_{3}\int_{0}^{L}f_{x}^{4}\overline{u^{3}}dx-\varepsilon _{3}\int_{0}^{L}f^{5}%
\overline{u^{3}}dx+\gamma \int_{0}^{L}f_{x}^{1}\overline{u^{3}}dx
\label{CASE4-lemma2-EQ5-Step2}
\end{equation}%
Again, using the compatibility condition \eqref{compatibility} in %
\eqref{CASE4-lemma2-EQ5-Step2}, we get \eqref{CASE4-lemma2-EQ1-Step2}.%
\newline
\textbf{Step 3.} The aim of this step is to prove \eqref{CASE4-lemma2-EQ1}.
For this aim adding \eqref{CASE4-lemma2-EQ1-Step1} and %
\eqref{CASE4-lemma2-EQ1-Step2} and taking the real part, we get 
\begin{equation*}
\begin{array}{c}
\displaystyle\alpha \int_{0}^{L}\abs{v_x}^{2}dx+\varepsilon _{3}\int_{0}^{L}%
\abs{u^3}^{2}dx=\rho \int_{0}^{L}\abs{z}^{2}dx+\Re \left( \rho \int_{0}^{L}z%
\overline{f^{1}}dx\right) -\Re \left( a\int_{0}^{L}z\overline{v}dx\right)
+\Re \left( \rho \int_{0}^{L}f^{2}\overline{v}dx\right) \\[0.1in] 
\displaystyle+\Re \left( b^{-1}\xi \varepsilon _{3}\int_{0}^{L}f_{x}^{4}%
\overline{u^{3}}dx\right) -\Re \left( b^{-1}\varepsilon _{3}\int_{0}^{L}f^{5}%
\overline{u^{3}}dx\right) +\Re \left( \gamma b^{-1}\int_{0}^{L}f_{x}^{1}%
\overline{u^{3}}dx\right) .%
\end{array}%
\end{equation*}%
It follows that 
\begin{equation}
\begin{array}{c}
\displaystyle\alpha \int_{0}^{L}\abs{v_x}^{2}dx+\varepsilon _{3}\int_{0}^{L}%
\abs{u^3}^{2}dx\leq \rho \int_{0}^{L}\abs{z}^{2}dx+\rho \int_{0}^{L}\abs{z}%
\abs{f^1}dx+a\int_{0}^{L}\abs{z}\abs{v}dx+\rho \int_{0}^{L}\abs{f^2}\abs{v}dx
\\[0.1in] 
\displaystyle+b^{-1}\xi \varepsilon _{3}\int_{0}^{L}\abs{f_x^4}\abs{u^3}%
dx+b^{-1}\varepsilon _{3}\int_{0}^{L}\abs{f^5}\abs{u^3}dx+\gamma
b^{-1}\int_{0}^{L}\abs{f_x^1}\abs{u^3}dx.%
\end{array}
\label{CASE4-lemma2-EQ1-Step3}
\end{equation}%
Using the facts that, $\sqrt{\rho }\Vert z\Vert \leq \Vert U\Vert _{\mathcal{%
H}}$, $\sqrt{\varepsilon _{3}}\Vert u^{3}\Vert \leq \Vert U\Vert _{\mathcal{H%
}}$, $\sqrt{\alpha }\Vert f_{x}^{1}\Vert \leq \Vert F\Vert _{\mathcal{H}}$, $%
\sqrt{\varepsilon _{3}}\Vert f^{5}\Vert \leq \Vert F\Vert _{\mathcal{H}}$
and Poincar\'{e} inequality, we get 
\begin{equation}
\left\{ 
\begin{array}{l}
\displaystyle\rho \int_{0}^{L}\abs{z}\abs{f^1}dx\leq \sqrt{\frac{\rho }{%
\alpha }}c_{p}\Vert U\Vert _{\mathcal{H}}\Vert F\Vert _{\mathcal{H}}, \\%
[0.1in] 
\displaystyle b^{-1}\varepsilon _{3}\int_{0}^{L}\abs{f^5}\abs{u^3}dx\leq
b^{-1}\Vert U\Vert _{\mathcal{H}}\Vert F\Vert _{\mathcal{H}}, \\[0.1in] 
\displaystyle\gamma b^{-1}\int_{0}^{L}\abs{f_x^1}\abs{u^3}dx\leq \frac{%
\gamma b^{-1}}{\sqrt{\alpha \varepsilon _{3}}}\Vert U\Vert _{\mathcal{H}%
}\Vert F\Vert _{\mathcal{H}}, \\ 
\displaystyle\rho \int_{0}^{L}\abs{f^2}\abs{v}dx\leq \sqrt{\frac{\rho }{%
\alpha }}c_{p}\Vert U\Vert _{\mathcal{H}}\Vert F\Vert _{\mathcal{H}}.%
\end{array}%
\right.  \label{CASE4-lemma2-EQ2-Step3}
\end{equation}%
Applying Cauchy-Schwarz and Young inequality, and using %
\eqref{CASE4-Lemm1-EQ1}, we get 
\begin{equation}
a\int_{0}^{L}\abs{z}\abs{v}dx\leq \frac{a}{2r_{1}}\int_{0}^{L}\abs{z}^{2}dx+%
\frac{ar_{1}c_{p}^{2}}{2}\int_{0}^{L}\abs{v_x}^{2}dx\leq \frac{1}{2r_{1}}%
\Vert U\Vert _{\mathcal{H}}\Vert F\Vert _{\mathcal{H}}+\frac{ar_{1}c_{p}^{2}%
}{2}\int_{0}^{L}\abs{v_x}^{2}dx.  \label{CASE4-lemma2-EQ3-Step3}
\end{equation}%
Since $F$ in $\mathcal{H}$, then the components of $F$ satisfies the
compatibility condition \eqref{compatibility}, then we get 
\begin{equation}
\xi f_{x}^{4}-f^{5}+\frac{\gamma }{\varepsilon _{3}}f_{x}^{1}=0.
\label{CASE4-lemma2-EQ4-Step3}
\end{equation}%
Using \eqref{CASE4-lemma2-EQ2-Step3}, \eqref{CASE4-lemma2-EQ4-Step3}, $\sqrt{%
\varepsilon _{3}}\Vert u^{3}\Vert \leq \Vert U\Vert _{\mathcal{H}}$, $\sqrt{%
\varepsilon _{3}}\Vert f^{5}\Vert \leq \Vert F\Vert _{\mathcal{H}}$ and $%
\sqrt{\alpha }\Vert f_{x}^{1}\Vert \leq \Vert F\Vert _{\mathcal{H}}$, we get 
\begin{equation}
b^{-1}\xi \varepsilon _{3}\int_{0}^{L}\abs{f_x^4}\abs{u^3}dx\leq
b^{-1}\varepsilon _{3}\left( \Vert f^{5}\Vert +\frac{\gamma }{\varepsilon
_{3}}\Vert f_{x}^{1}\Vert \right) \Vert u^{3}\Vert \leq b^{-1}\left( 1+\frac{%
\gamma }{\sqrt{\varepsilon _{3}\alpha }}\right) \Vert U\Vert _{\mathcal{H}%
}\Vert F\Vert _{\mathcal{H}}.  \label{CASE4-lemma2-EQ5-Step3}
\end{equation}%
Inserting \eqref{CASE4-lemma2-EQ2-Step3}, \eqref{CASE4-lemma2-EQ3-Step3} and %
\eqref{CASE4-lemma2-EQ5-Step3}, in \eqref{CASE4-lemma2-EQ1-Step3} 
\begin{equation*}
\left( \alpha -\frac{ar_{1}c_{p}^{2}}{2}\right) \int_{0}^{L}\abs{v_x}%
^{2}dx+\varepsilon _{3}\int_{0}^{L}\abs{u^3}^{2}dx\leq \left[ 2\left( \sqrt{%
\frac{\rho }{\alpha }}c_{p}+b^{-1}+\frac{\gamma b^{-1}}{\sqrt{\alpha
\varepsilon _{3}}}+\frac{1}{4r_{1}}\right) +\mathcal{S}_{1}\right] \Vert
U\Vert _{\mathcal{H}}\Vert F\Vert _{\mathcal{H}}.
\end{equation*}%
Taking $r_{1}=\frac{\alpha }{ac_{p}^{2}}$ in the above inequality, we get %
\eqref{CASE4-lemma2-EQ1}. The proof has been completed.
\end{proof}

%%%%%%%%%%%%%%%%%%%%%%%%%%%%%%%%%%%%%%%%%%

\begin{lemma}
\label{Case4-Lemm4} The solution $(v,z,u^1,u^2,u^3)\in D(\mathcal{A}%
_{a,b,0}) $ of equation \eqref{CASE4-EXPO-1}satisfies the following
estimation: 
\begin{equation}  \label{CASE4-Lemm4-EQ1}
\mu\int_0^L\abs{u^1}^2dx\leq \mathcal{S}_4\|U\|_{\mathcal{H}}\|F\|_{\mathcal{%
H}},
\end{equation}
where $\mathcal{S}_4=\xi\varepsilon_3\left(\left(1+\frac{b^2\xi\varepsilon_3%
}{2\mu}\right)\mathcal{S}_2+\left[\left(\frac{1}{\xi}+\frac{\gamma^2}{%
2\xi^2\varepsilon_3^2}\right)\varepsilon_3^{-1}+\frac{1}{\alpha}\right]%
\mathcal{S}_3+\frac{2}{\sqrt{\xi\varepsilon_3\mu}}\right)$.
\end{lemma}

\begin{proof}
By proceeding the same technics used in Lemma \ref{LEMM3}, we get %
\eqref{CASE4-Lemm4-EQ1}. The proof has been completed.
\end{proof}

%%%%%%%%%%%%%%%%%%%%%%%%%%%%%%%%%%%%%%%%%%
%%%%%%%%%%%%%%%%%%%%%%%%%%%%%%%%%%%%%%%%%%
\noindent \textbf{Proof of Proposition \ref{CASE4-PRO3}}. From %
\eqref{CASE4-Lemm1-EQ1}, \eqref{CASE4-Lemm1-EQ2}, \eqref{CASE4-lemma2-EQ1}
and \eqref{CASE4-Lemm4-EQ1}, we get 
\begin{equation*}
\|U\|^2_{\mathcal{H}}=\alpha\|v_x\|^2+\rho\|z\|^2+\mu\|u^1\|^2+\xi%
\varepsilon_3\|u^2\|^2+\varepsilon_3\|u^3\|^2\leq \mathcal{S}\|U\|_{\mathcal{%
H}}\|F\|_{H},
\end{equation*}
where $\mathcal{S}=\rho S_1+\xi \varepsilon_3\mathcal{S}_2+3\mathcal{S}_3+%
\mathcal{S}_4$. The proof has been completed.\\[0.1in]
%%%%%%%%%%%%%%%%%%%%%%%%%%%%%%%%%%%%%%%%%%
%%%%%%%%%%%%%%%%%%%%%%%%%%%%%%%%%%%%%%%%%%
\noindent \textbf{Proof of Theorem \ref{THEOREM1-EXP-Case4}} For all $U\in D(%
\mathcal{A}_{a,b,0})$ according to Proposition \eqref{CASE4-PRO3}, we get 
\begin{equation*}
\|U\|_{\mathcal{H}}\leq \mathcal{S}\|(i\la I-\mathcal{A}_{a,0,c})U\|_{%
\mathcal{H}}.
\end{equation*}
Thus, we have 
\begin{equation*}
\|(i\la I-\mathcal{A}_{a,b,0})^{-1}V\|_{\mathcal{H}}\leq \mathcal{S}\|V\|_{%
\mathcal{H}},\quad \forall V\in \mathcal{H}.
\end{equation*}
Therefore, from the above equation, we get \eqref{CASE4-Condition2-EXP}
holds. Thus, we get the conclusion by applying Huang and Pr\"uss Theorem. 
%%%%%%%%%%%%%%%%%%%%%%%
%%%%%%%%%%%%%%%%%%%%%%% 

\section{The stretching of the centreline of the beam in $x-$direction only
is damped and "$a\neq 0\ \text{and}\ (b,c)=(0,0)$"}

\noindent In this section, we prove that the Lorenz system with only one
damping acting on the stretching of the centreline still be exponentially
stable. The main result of this section is the following theorem:

\begin{theoreme}
\label{EXP-CASE7} Assume that $a\neq 0$ and $(b,c)=(0,0)$. Then, the $C_0-$%
semigroup of contraction $e^{t\mathcal{A}}$ is exponentially stable; i.e.
there exists constants $M\geq 1$ and $\epsilon>0$ independent of $U_0$ such
that 
\begin{equation*}
\|e^{t\mathcal{A}_{a,0,0}}U_0\|_{\mathcal{H}}\leq Me^{-\epsilon t}\|U_0\|_{%
\mathcal{H}}. 
\end{equation*}
\end{theoreme}

\noindent According to Huang and Pr\"uss, we have to check if the following
conditions hold: 
\begin{equation}  \label{Condition1-EXP}
i\mathbb{R}\subset \rho(\mathcal{A}_{a,0,0})  \tag{${\rm H1}$}
\end{equation}
and 
\begin{equation}  \label{CASE7-Condition2-EXP}
\sup_{\la \in \mathbb{R}} \|(i\la I-\mathcal{A}_{a,0,0})^{-1}\|_{\mathcal{L}(%
\mathcal{H})}=O(1).  \tag{${\rm H7}$}
\end{equation}
Since $i\mathbb{R}\subset \rho(\mathcal{A}_{a,0,0})$, then condition %
\eqref{Condition1-EXP} is satisfied. We will prove condition %
\eqref{CASE7-Condition2-EXP} by a contradiction argument. For this purpose,
suppose that \eqref{CASE7-Condition2-EXP} is false, then there exists $%
\left\{(\la^n,U^n)\right\}_{n\geq 1}\subset \mathbb{R}^{\ast}\times D(%
\mathcal{A})$ with 
\begin{equation}  \label{CASE7-CONTRA1}
\abs{\la^n}\to \infty\quad \text{and}\quad \|U^n\|_{\mathcal{H}%
}=\|(v_n,z_n,u^1_n,u^2_n,u^3_n)^{\top}\|_{\mathcal{H}}=1,
\end{equation}
such that 
\begin{equation}  \label{CASE7-CONTRA2}
\left(i\la^nI-\mathcal{A}_{a,0,0}%
\right)U^n=F^n:=(f^1_n,f^2_n,f^3_n,f^4_n,f^5_n)\to 0\quad \text{in}\quad 
\mathcal{H}.
\end{equation}
For simplicity, we drop the index $n$. Equivalenty, from %
\eqref{CASE7-CONTRA2}, we have 
\begin{eqnarray}
i\la v-z&=&f^1,  \label{CASE7-EQ1-EXP1} \\
i\la z-\frac{\alpha}{\rho}v_{xx}-\frac{\gamma}{\rho}u^3_x+\frac{a}{\rho}%
z&=&f^2,  \label{CASE7-EQ2-EXP1} \\
i\la u^1-u^2+u^3_x&=&f^3,  \label{CASE7-EQ3-EXP1} \\
i\la u^2+\frac{\mu}{\xi \varepsilon_3}u^1&=&f^4,  \label{CASE7-EQ4-EXP1} \\
i\la u^3+\frac{\mu}{\varepsilon_3}u^1_x-\frac{\gamma}{\varepsilon_3}%
z_x&=&f^5.  \label{CASE7-EQ5-EXP1}
\end{eqnarray}
Here we will check the condition \eqref{CASE7-CONTRA2} by finding a
contradiction with \eqref{CASE7-Condition2-EXP} by showing that $\|U\|_{%
\mathcal{H}}=o(1)$. For clarity, we divide the proof into several Lemmas. 
%%%%%%%%%%%%%%%%%%%%%%%% Lemma 1  %%%%%%%%%%%%%%%%%%%%%%%% 

\begin{lemma}
\label{LEMM1-CASE7} The solution $(v,z,u^1,u^2,u^3)\in D(\mathcal{A}%
_{a,0,0}) $ of equation \eqref{CASE7-CONTRA2} satisfies the following
estimates: 
\begin{equation}  \label{LEMM1-CASE7-EQ1}
\int_0^L\abs{z}^2dx=o(1)\quad \text{and}\quad \int_0^L\abs{\la v}^2dx=o(1).
\end{equation}
\end{lemma}

\begin{proof}
First, taking the inner product of \eqref{CASE7-CONTRA2} with $U$ in $%
\mathcal{H}$, we obtain 
\begin{equation}  \label{LEMM1-CASE7-EQ2}
a\int_0^L\abs{z}^2=-\Re\left(\mathcal{A}_{a,b;c}U,F\right)_{\mathcal{H}%
}=\Re\left(F,U\right)_{\mathcal{H}}\leq \|U\|_{ \mathcal{H}}\|F\|_{\mathcal{H%
}}.
\end{equation}
Thus, from the above estimation and the fact that $\|F\|_{\mathcal{H}}=o(1)$
and $\|U\|_{\mathcal{H}}=1$, we obtain the first estimation in %
\eqref{LEMM1-CASE7-EQ1}. From \eqref{CASE7-EQ1-EXP1}, we deduce that 
\begin{equation}  \label{LEMM1-CASE7-EQ3}
\int_0^L\abs{\la v}^2dx\lesssim \int_0^L\abs{z}^2dx+\int_0^L\abs{f^1}^2dx.
\end{equation}
Finally, from \eqref{LEMM1-CASE7-EQ3}, the first estimation in %
\eqref{LEMM1-CASE7-EQ1}, we get the second estimation in %
\eqref{LEMM1-CASE7-EQ1}. The proof is thus complete.
\end{proof}

%%%%%%%%%%%%%%%%%%%%%%%% Lemma 2  %%%%%%%%%%%%%%%%%%%%%%%%

\begin{lemma}
\label{LEMM2-CASE7} The solution $(v,z,u^1,u^2,u^3)\in D(\mathcal{A}%
_{a,0,0}) $ of equation \eqref{CASE7-CONTRA2} satisfies the following
estimates: 
\begin{equation}  \label{LEMM2-CASE7-EQ0}
\int_0^L\abs{v_x}^2dx=o(1).
\end{equation}
\end{lemma}

\begin{proof}
Multiplying \eqref{CASE7-EQ2-EXP1} by $\rho \overline{v}$ integrating by
parts over $(0,L)$, we get 
\begin{equation}  \label{LEMM2-CASE7-EQ1}
i\la \rho\int_0^Lz\overline{v}dx+\alpha\int_0^L\abs{v_x}^2dx+\gamma%
\int_0^Lu^3\overline{v_x}dx+a\int_0^Lz\overline{v}dx=\rho\int_0^Lf^2%
\overline{v}dx.
\end{equation}
Using Lemma \ref{LEMM1-CASE7} and the fact that $\|f^2\|_{L^2(0,L)}=o(1)$,
we get 
\begin{equation}  \label{LEMM2-CASE7-EQ2}
\left|i\la \rho\int_0^Lz\overline{v}dx\right|=o(1),\quad \left|a\int_0^Lz%
\overline{v}dx\right|=o(\abs{\la}^{-1})\quad \text{and}\quad
\left|\rho\int_0^Lf^2\overline{v}dx\right|=o(\abs{\la}^{-1}).
\end{equation}
Inserting \eqref{LEMM2-CASE7-EQ2} in \eqref{LEMM2-CASE7-EQ1}, we get 
\begin{equation}  \label{LEMM2-CASE7-EQ3}
\alpha\int_0^L\abs{v_x}^2dx+\gamma\int_0^Lu^3\overline{v_x}dx=o(1).
\end{equation}
Now, multiplying \eqref{CASE7-EQ3-EXP1} by $\gamma\overline{v}$ integrating
by parts over $(0,L)$, we get 
\begin{equation}  \label{LEMM2-CASE7-EQ4}
i\la \gamma\int_0^Lu^1\overline{v}dx-\gamma\int_0^Lu^2\overline{v}%
dx-\gamma\int_0^Lu^3\overline{v_x}dx+\gamma u^3(L)\overline{v}%
(L)=\gamma\int_0^Lf^3\overline{v}dx.
\end{equation}
Using the facts that $u^1$, $u^2$ are uniformly bounded in $L^2(0,L)$,
equation \eqref{LEMM1-CASE7-EQ1} and the fact that $\|f^3\|_{H_0^1(0,L)}=o(1)
$, we get 
\begin{equation}  \label{LEMM2-CASE7-EQ5}
\left|i\la \gamma\int_0^Lu^1\overline{v}dx\right|=o(1),\quad
\left|\gamma\int_0^Lu^2\overline{v}dx\right|=o(\abs{\la}^{-1})\quad \text{and%
}\quad \left|\gamma\int_0^Lf^3\overline{v}dx\right|=o(\abs{\la}^{-1}).
\end{equation}
From \eqref{CASE7-EQ3-EXP1}, it is easy to see that $\|u^3_x|_{L^2(0,L)}=O%
\left(\abs{\la}\right)$. Using Galgliardo-Nirenberg inequality, $\|u^3\|$
and $\|v_x\|$ are uniformly bounded in $L^2(0,L)$ and \eqref{LEMM1-CASE7-EQ1}%
, we get 
\begin{equation}  \label{LEMM2-CASE7-EQ6}
\left\{%
\begin{array}{l}
\abs{u^3(L)}\lesssim \|u^3_x\|^{\frac{1}{2}}\|u^3\|^{\frac{1}{2}%
}+\|u^3\|\lesssim O(\abs{\la}^{\frac{1}{2}}), \\[0.1in] 
\abs{v(L)}\lesssim \|v_x\|^{\frac{1}{2}}\|v\|^{\frac{1}{2}}+\|v\|\lesssim o(%
\abs{\la}^{-\frac{1}{2}}).%
\end{array}
\right.
\end{equation}
Inserting \eqref{LEMM2-CASE7-EQ5} and \eqref{LEMM2-CASE7-EQ6} in %
\eqref{LEMM2-CASE7-EQ4}, we get 
\begin{equation*}
\left|\gamma\int_0^Lu^3\overline{v_x}dx\right|=o(1).
\end{equation*}
Inserting the above estimation in \eqref{LEMM2-CASE7-EQ3}, we get %
\eqref{LEMM2-CASE7-EQ0}. The proof is thus completed.
\end{proof}

%%%%%%%%%%%%%%%%%%%%%%%% Lemma 3  %%%%%%%%%%%%%%%%%%%%%%%%

\begin{lemma}
\label{LEMM3-CASE7} The solution $(v,z,u^1,u^2,u^3)\in D(\mathcal{A}%
_{a,0,0}) $ of equation \eqref{CASE7-CONTRA2} satisfies the following
estimates: 
\begin{equation}  \label{LEMM3-CASE7-EQ1}
\int_0^L\abs{u^2_x}^2dx=o(1)\quad \text{and}\quad \int_0^L\abs{u^2}^2dx=o(1).
\end{equation}
\end{lemma}

\begin{proof}
Multiplying \eqref{CASE7-EQ2-EXP1} by $\overline{u^2}$, integrating by parts
over $0,L)$, we get 
\begin{equation}  \label{LEMM3-CASE7-EQ2}
i\la \int_0^Lz\overline{u^2}dx+\frac{\alpha}{\rho}\int_0^Lv_x\overline{u^2_x}%
dx+\frac{\gamma}{\rho}\int_0^Lu^3\overline{u^2_x}dx+\frac{a}{\rho}\int_0^Lz%
\overline{u^2}dx=\int_0^Lf^2\overline{u^2}dx
\end{equation}
From the compatibility condition, we have 
\begin{equation*}
u^3=\xi u^2_x+\frac{\gamma}{\varepsilon_3}v_x. 
\end{equation*}
Inserting the above equation in \eqref{LEMM3-CASE7-EQ2}, we get 
\begin{equation}  \label{LEMM3-CASE7-EQ3}
i\la \int_0^Lz\overline{u^2}dx+\left(\frac{\alpha}{\rho}+\frac{\gamma^2}{%
\rho\varepsilon_3}\right)\int_0^Lv_x\overline{u^2_x}dx+\frac{\gamma\xi}{\rho}%
\int_0^L\abs{u^2_x}^2dx+\frac{a}{\rho}\int_0^Lz\overline{u^2}dx=\int_0^Lf^2%
\overline{u^2}dx.
\end{equation}
Using the facts that $u^2_x$ and $u^2$ are uniformly bounded in $L^2(0,L)$, %
\eqref{LEMM1-CASE7-EQ1}, \eqref{LEMM2-CASE7-EQ0} and $\|f^2\|_{L^2(0,L)}=o(1)
$, we get 
\begin{equation*}
\left|\int_0^Lv_x\overline{u^2_x}dx\right|=o(1),\quad \left|\int_0^Lz%
\overline{u^2}dx\right|=o(1)\text{and}\quad \left|\int_0^Lf^2\overline{u^2}%
dx\right|=o(1).
\end{equation*}
Inserting the above estimations in \eqref{LEMM3-CASE7-EQ3}, we get 
\begin{equation}  \label{LEMM3-CASE7-EQ4}
i\la \int_0^Lz\overline{u^2}dx+\frac{\gamma\xi}{\rho_3}\int_0^L\abs{u^2_x}%
^2dx=o(1).
\end{equation}
Multiplying \eqref{CASE7-EQ4-EXP1} by $\overline{z}$ integrating over $(0,L)$%
, using the fact that $u^1$ is uniformly bounded in $L^2(0,L)$, $%
\|f_4\|_{L^2(0,L)}=o(1)$ and \eqref{LEMM1-CASE7-EQ1}, we get we get 
\begin{equation*}
i\la \int_0^Lu^2\overline{z}dx+\underbrace{\frac{\mu}{\xi\varepsilon_3}%
\int_0^Lu^1\overline{z}dx}_{o(1)}=\underbrace{\int_0^Lf^4\overline{z}dx}%
_{o(1)}.
\end{equation*}
Inserting the above estimation in \eqref{LEMM3-CASE7-EQ4}, we get the first
estimation in \eqref{LEMM3-CASE7-EQ1}. Using the first estimation in %
\eqref{LEMM3-CASE7-EQ1} and Poincr\'e inequality, we get the second
estimation in \eqref{LEMM3-CASE7-EQ1}. The proof has been completed.
\end{proof}

%%%%%%%%%%%%%%%%%%%%%%%% Lemma 4 %%%%%%%%%%%%%%%%%%%%%%%%

\begin{lemma}
\label{LEMM4-CASE7} The solution $(v,z,u^1,u^2,u^3)\in D(\mathcal{A}%
_{a,0,0}) $ of equation \eqref{CASE7-CONTRA2} satisfies the following
estimates: 
\begin{equation}  \label{LEMM4-CASE7-EQ1}
\int_0^L\abs{u^3}^2dx=o(1)\quad \text{and}\quad \int_0^L\abs{u^1}^2dx=o(1).
\end{equation}
\end{lemma}

\begin{proof}
First, we prove the first estimation in \eqref{LEMM4-CASE7-EQ1}. For this
aim, Using the compatibility condition \eqref{compatibility}, %
\eqref{LEMM3-CASE7-EQ1} and \eqref{LEMM2-CASE7-EQ0}, we get 
\begin{equation}  \label{LEMM4-CASE7-EQ2}
\int_0^L\abs{u^3}^2dx\leq 2\xi^2\int_0^L\abs{u^2_x}^2dx+2\frac{\gamma^2}{%
\varepsilon_3^2}\int_0^L\abs{v_x}^2dx\leq o(1).
\end{equation}
Now, we prove the second estimation in \eqref{LEMM4-CASE7-EQ1}. For this
aim, multiplying \eqref{CASE7-EQ4-EXP1} by $\overline{u^1}$, integrating
over $(0,L)$ and using the facts that $u^1$ is uniformly bounded in $L^2(0,L)
$ and $\|f^4\|_{L^2(0,L)}=o(1)$, we get 
\begin{equation}  \label{LEMM4-CASE7-EQ3}
i\la \int_0^Lu^2\overline{u^1}dx+\frac{\mu}{\xi\varepsilon_3}\int_0^L%
\abs{u^1}^2dx=o(1).
\end{equation}
Now, multiplying \eqref{CASE7-EQ3-EXP1} by $\overline{u^2}$, integrating by
parts over $(0,L)$ and using the fact that $\|f_3\|_{L^2(0,L)}=o(1)$, %
\eqref{LEMM3-CASE7-EQ1} and the first estimation in \eqref{LEMM4-CASE7-EQ1},
we get 
\begin{equation}  \label{LEMM4-CASE7-EQ4}
i\la \int_0^Lu^1\overline{u^2}dx-\underbrace{\int_0^Lu^3\overline{u_x^2}dx}%
_{o(1)}=o(1).
\end{equation}
Inserting \eqref{LEMM3-CASE7-EQ4} in \eqref{LEMM3-CASE7-EQ3}, we get the
second estimation in \eqref{LEMM3-CASE7-EQ1}. The proof has been completed.
\end{proof}

%%%%%%%%%%%%%%%%%%%%%%%%%%%%%%%%%%%%%%%%%%%%%%%% 
\noindent \textbf{Proof of Theorem \ref{EXP-CASE7}}. From Lemmas \ref%
{LEMM1-CASE7}-\eqref{LEMM4-CASE7}, we obtain $\|U\|_{\mathcal{H}}=o(1)$,
which contradicts \eqref{CASE7-CONTRA1}. This implies that 
\begin{equation*}
\sup_{\la \in \mathbb{R}}\|(i\la I-\mathcal{A}_{a,0,0})^{-1}\|_{\mathcal{H}%
}<\infty. 
\end{equation*}
Finally, according to Huang-Pruss theorem, we obtain the desired result. The
proof is thus complete.

\section{Numerical Results}

\noindent In this section, we will numerically illustrate the exponential
decay of the natural energy $E(t)$ associated to \eqref{Lorenz} system. To
carry out the numerical simulations, we first re-write the second-order
Lorenz system in a first-order form in time and then we discretize the
resulted system using a second-order centered finite difference
approximation for space and the second-order implicit backward
differentiation formula for time. The computational domain considered is $%
[0, 1]$ and the time interval is $[0, 100]$. For simplicity all the
parameter in Lorenz system are set to one. The following initial conditions
are used:

\begin{equation*}
(v,\phi,\theta,\eta)(\cdot,0)=(10^{-2}\sin{3\pi x},\cos{\pi x},\sin{\pi x}%
,\pi \cos{\pi x})
\end{equation*}

\begin{equation*}
(v_t,\phi_t,\theta_t,\eta_t)(\cdot,0)=(10^{2}\sin{3\pi x},0,0,0).\vspace{3mm}
\end{equation*}

\noindent Our results are presented in Figures \ref{case1}, \ref{case5} and \ref{Et}.
First, in figure \ref{case1} we show $v_t, v_x,\theta_t+\phi_x,\eta_t+\phi,
\theta-\eta_x$ as well as the natural energy $E(t)$ in the case where $a=0$, 
$b=0$ and $c=0$. The conservation of the natural energy is clearly shown in
this case. Then, we consider the following six cases:\newline

\begin{enumerate}
\item[$\mathbf{Case 1:}$] $(a,b,c)= (1,1,1)$.\\[-1mm]

\item[$\mathbf{Case 2:}$] $a=0$ and $(b,c)= (1,1)$.\\[-1mm]

\item[$\mathbf{Case 3:}$] $b=0$ and $(a,c)= (1,1)$.\\[-1mm]

\item[$\mathbf{Case 4:}$] $c=0$ and $(a,b)= (1,1)$.\\[-1mm]

\item[$\mathbf{Case 5:}$] $a=1$ and $(b,c)=(0,0)$.\\[-1mm]

%\item[$\mathbf{Case 6:}$] $b=1$ and $(a,c)=(0,0)$.\\[-1mm]

\item[$\mathbf{Case 6:}$] $c=1$ and $(a,b)=(0,0)$.\\[-1mm]
\end{enumerate}

\noindent The results for case 5, where $a=1$, $b=0$ and $c=0$, is presented
in figure \ref{case5}. As can be seen, we obtained an exponential decay of
the numerical solutions $v_t, v_x,\theta_t+\phi_x,\eta_t+\phi$, $%
\theta-\eta_x$ as well as the natural energy $E(t)$. This is consistent with
our theoretical results. For all the above mentioned cases, we obtained
numerical results similar to figure \ref{case5} showing an exponential decay
of the solutions as expected by our theoretical results. The figures are not
presented here to avoid repetition. However, we present the natural energy
for all case in figure \ref{Et}.

\begin{figure}[h]
\begin{center}
\begin{tabular}{ccc}
\includegraphics[width=5.5cm]{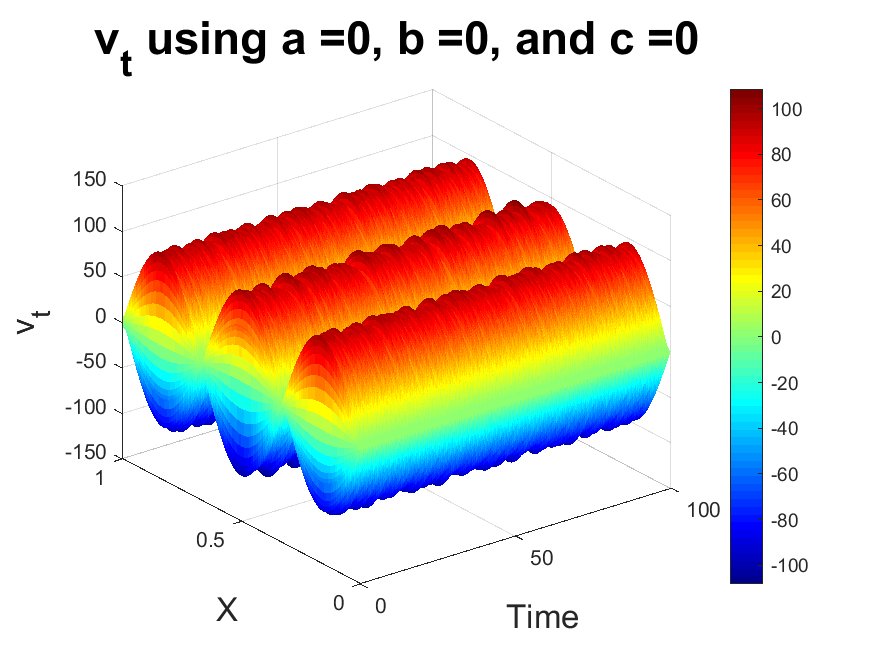} & %
\includegraphics[width=5.5cm]{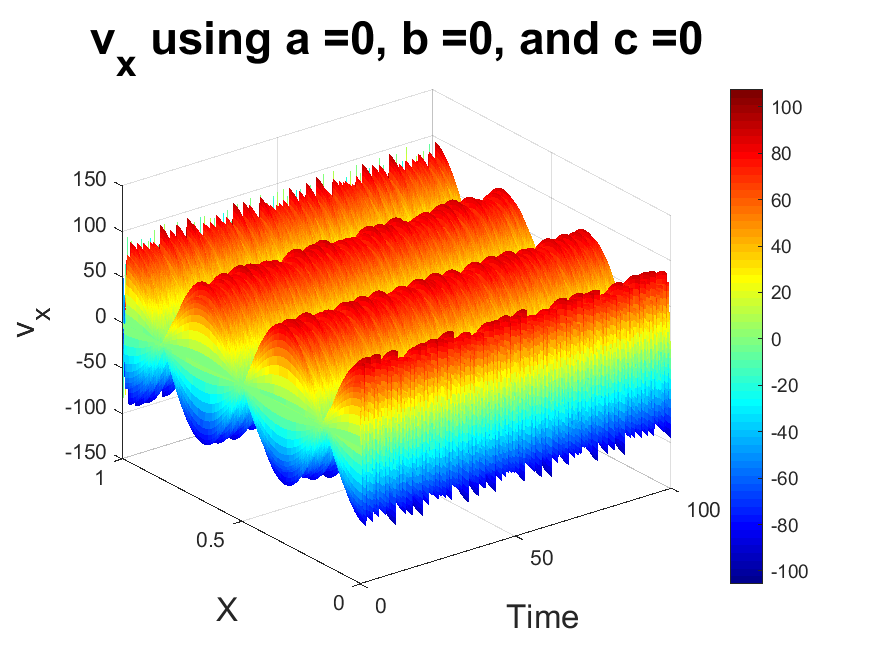} & %
\includegraphics[width=5.5cm]{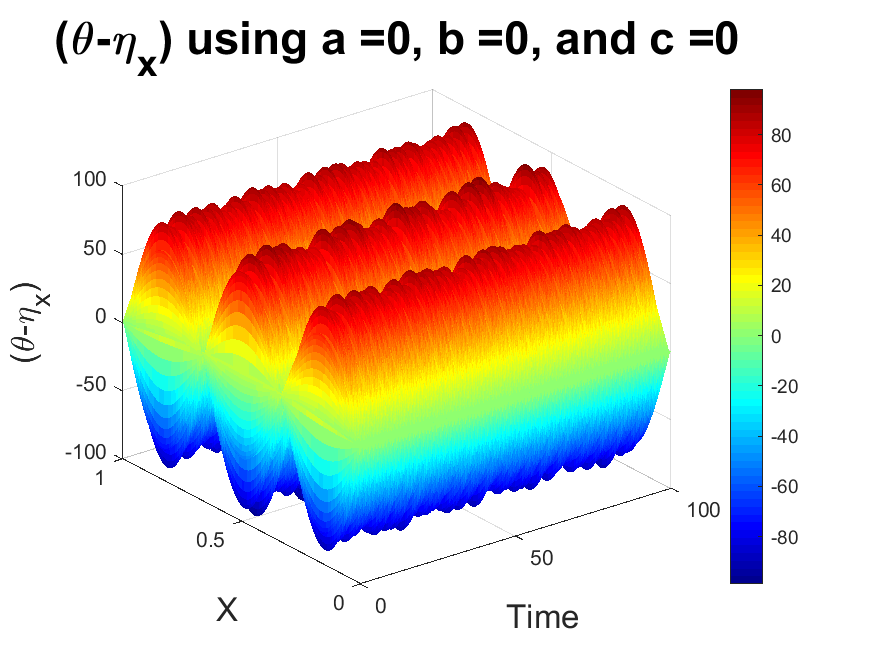} \\ 
\includegraphics[width=5.5cm]{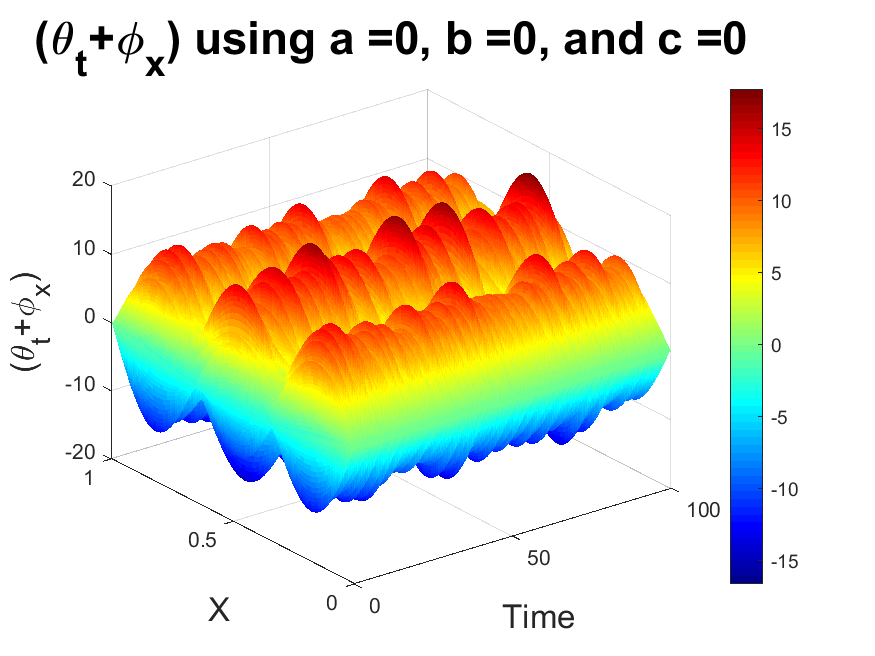} & %
\includegraphics[width=5.5cm]{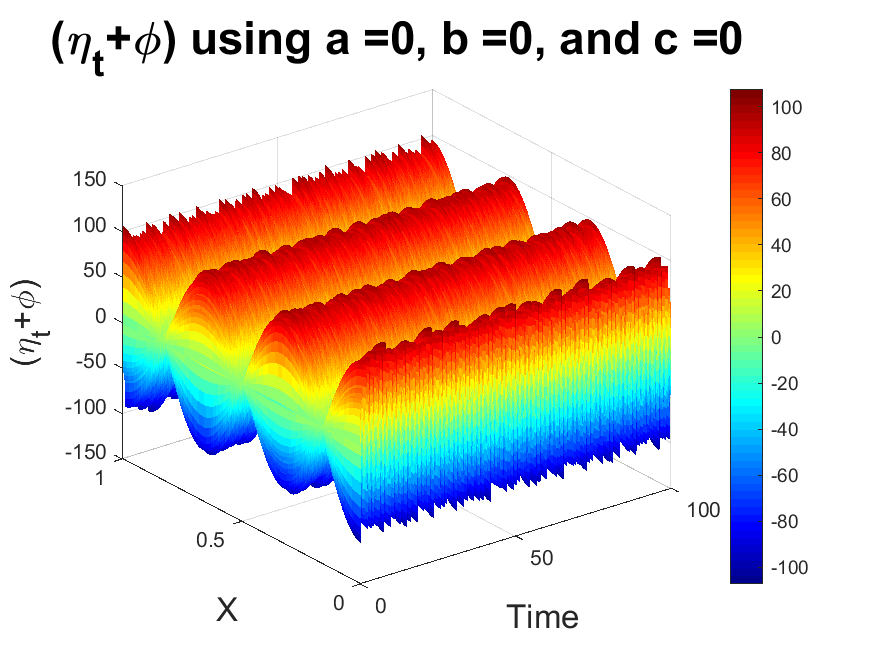} & %
\includegraphics[width=5.5cm]{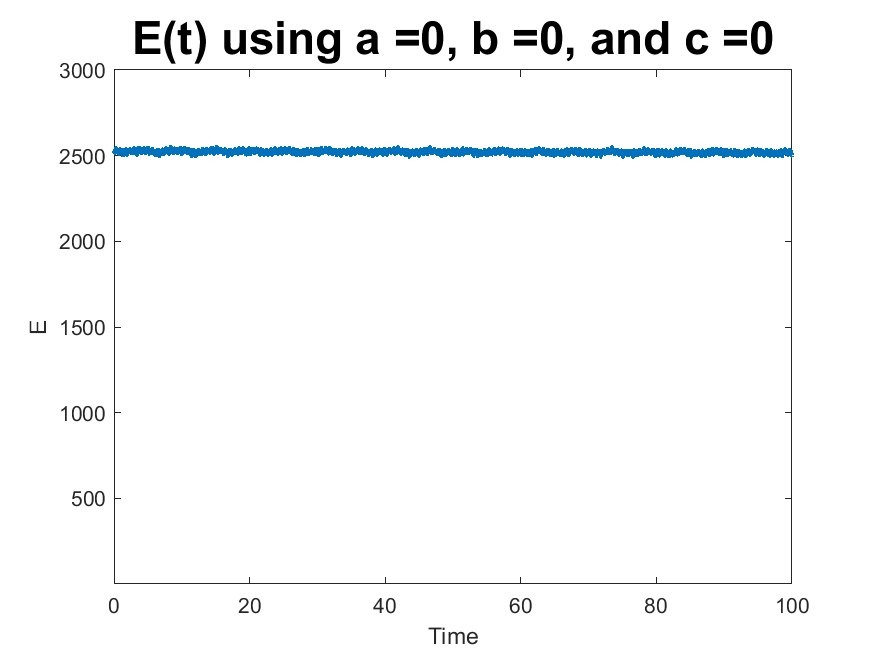} \\ 
&  & 
\end{tabular}
\end{center}
\caption{Space and time numerical solutions when $a=0$, $b=0$ and $c=0$. }
\label{case1}
\end{figure}

\begin{figure}[h]
\begin{center}
\begin{tabular}{ccc}
\includegraphics[width=5.5cm]{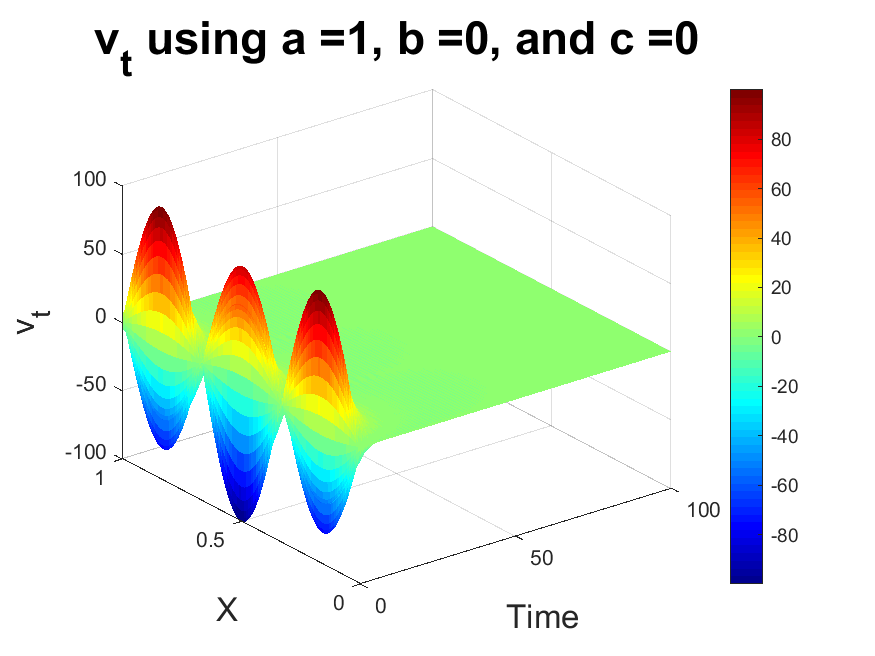} & %
\includegraphics[width=5.5cm]{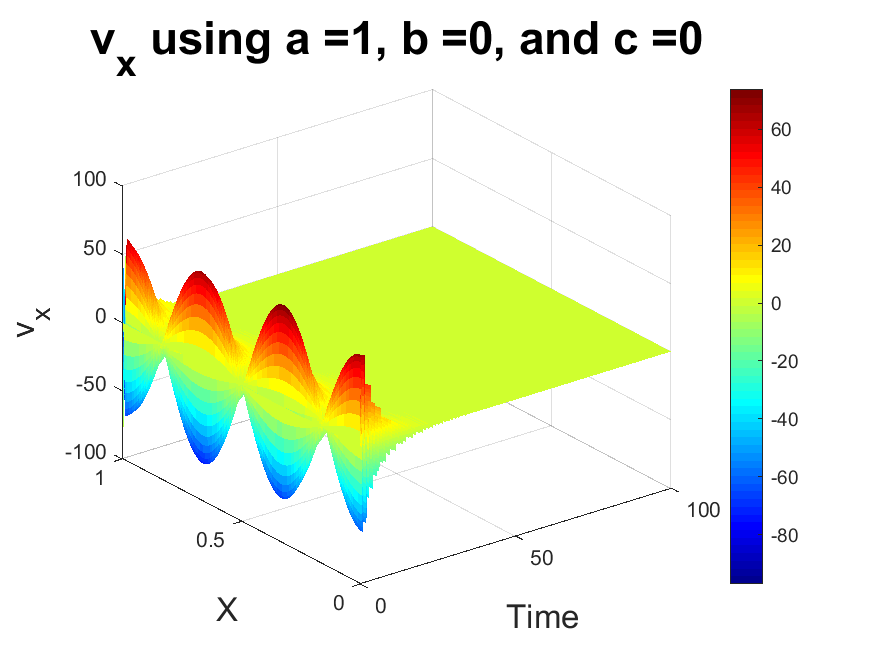} & %
\includegraphics[width=5.5cm]{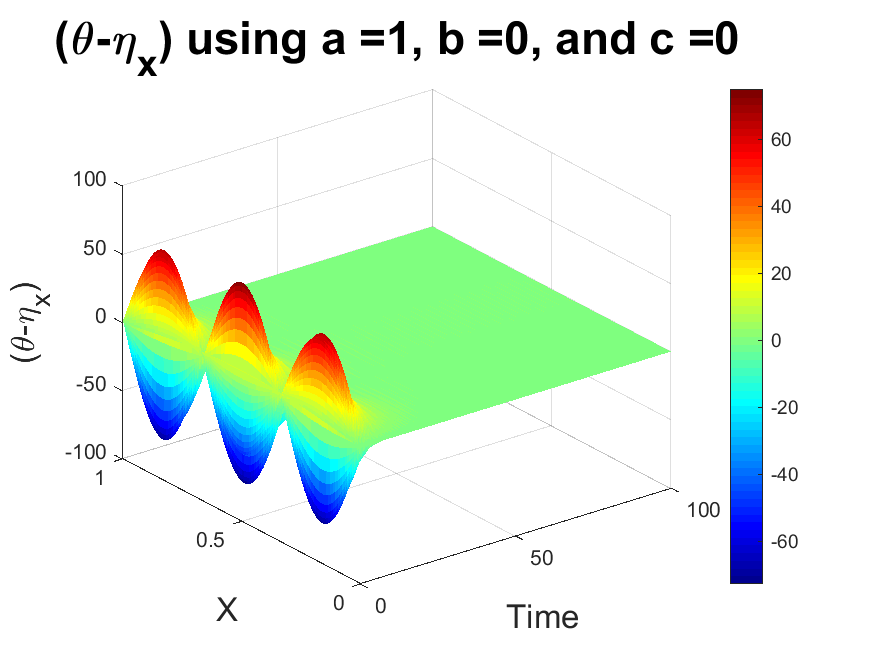} \\ 
\includegraphics[width=5.5cm]{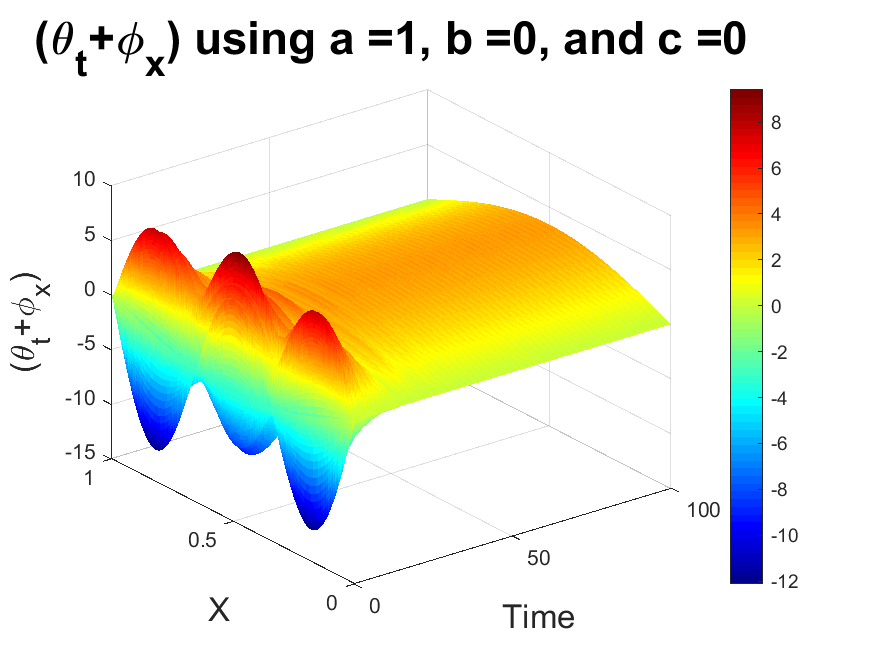} & %
\includegraphics[width=5.5cm]{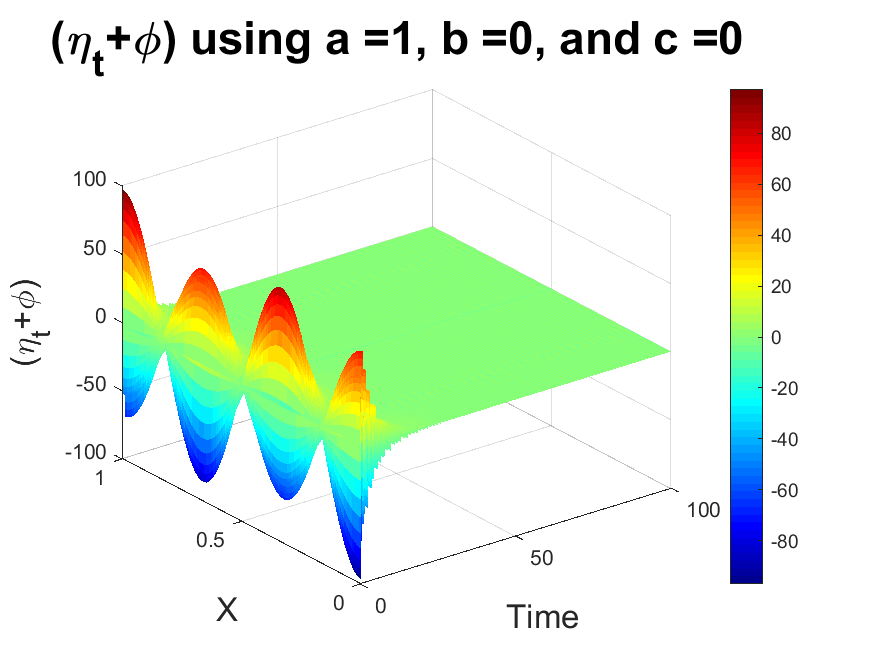} & %
\includegraphics[width=5.5cm]{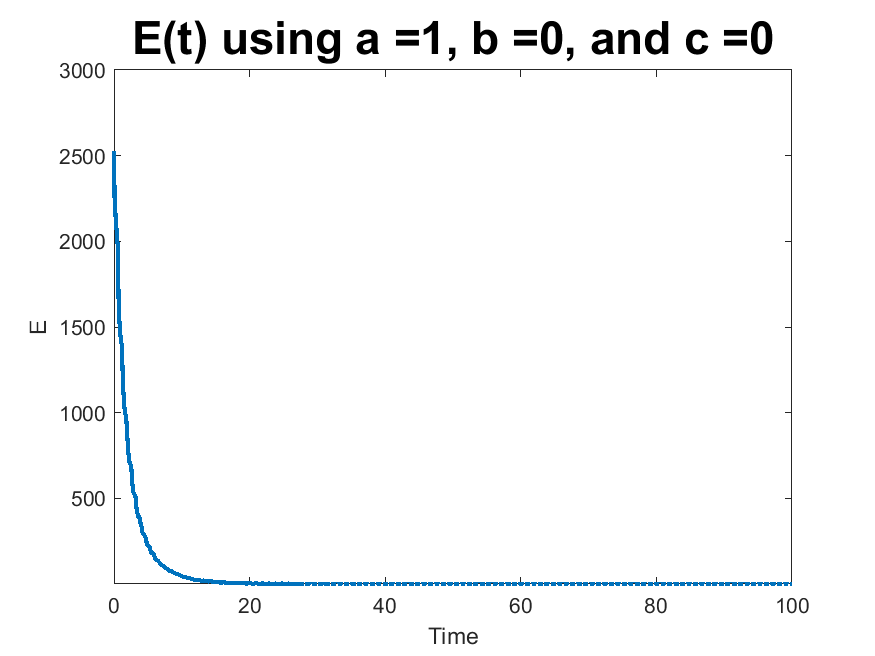} \\ 
&  & 
\end{tabular}
\end{center}
\caption{Space and time numerical solutions when $a=1$, $b=0$ and $c=0$
(Case 5).}
\label{case5}
\end{figure}

\begin{figure}[h]
\begin{center}
\begin{tabular}{ccc}
\includegraphics[width=5.5cm]{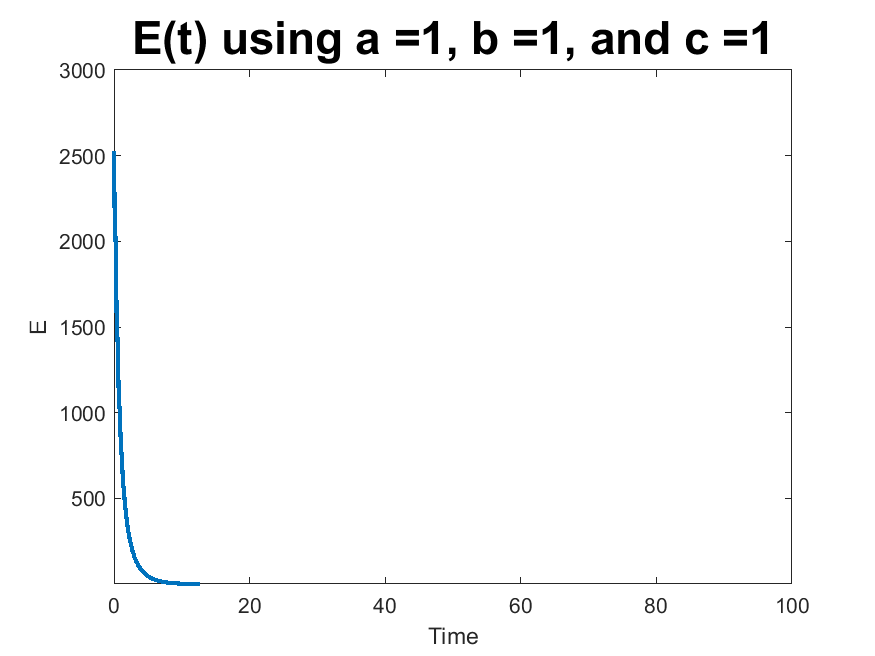} & %
\includegraphics[width=5.5cm]{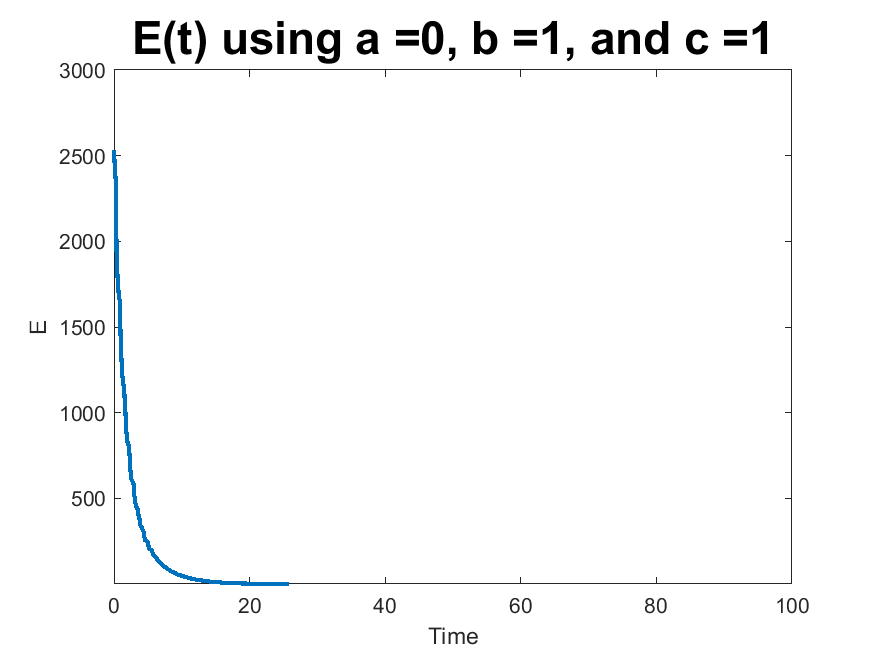} & %
\includegraphics[width=5.5cm]{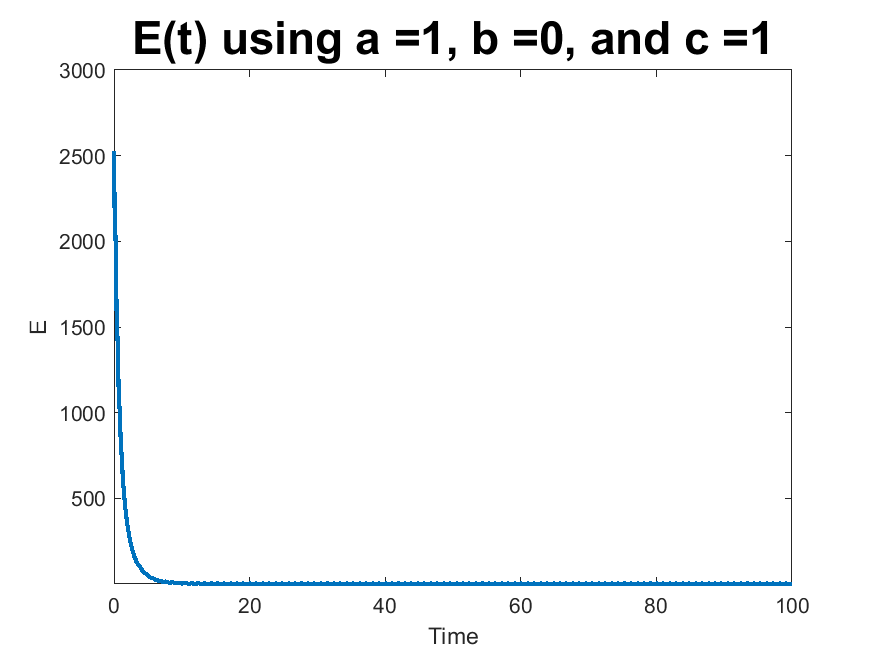} \\ 
\includegraphics[width=5.5cm]{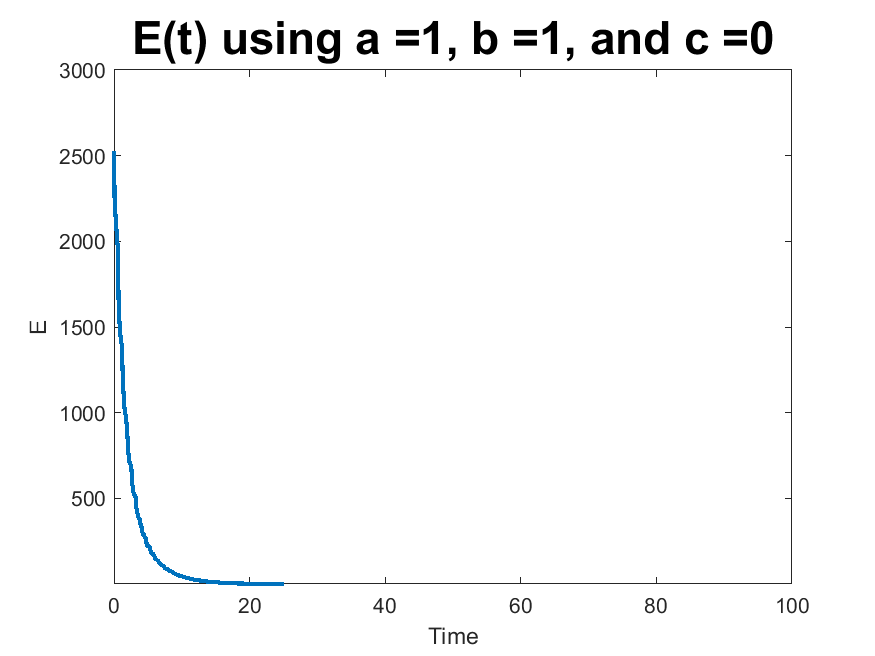} & %
\includegraphics[width=5.5cm]{Case_a_1_b_0_c_0_E.png} & %
\includegraphics[width=5.5cm]{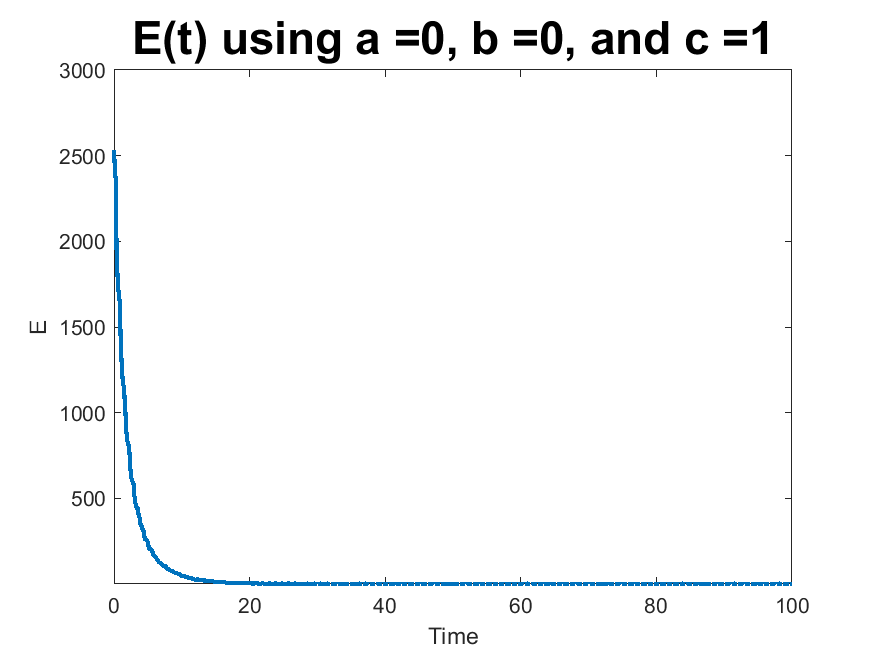} \\ 
&  & 
\end{tabular}
\end{center}
\caption{The natural energy $E(t)$ for all cases.}
\label{Et}
\end{figure}

$\newline
$ $\newline
$ $\newline
$ $\newline
$ $\newline
$

\section{Conclusion}

\noindent In this paper, we investigate the exponential stability of a
Lorenz Piezoelectric beam with partial viscous damping. Different cases has
been studied. we remark that it sufficient to controlled the stretching of
the centreline of the beam in $x-$direction to achieve the exponential
stability. The case where $b\neq 0$ and $(a,c)=(0,0)$ is still an open
problem,. However based on our numerical resulats we remark that we do not
obtain the exponential stability in the case where $b=1$ and $(a,c)=(0,0)$
(See Figure \ref{a=0b=1c=0}).

\begin{figure}[h!]
\begin{center}
\begin{tabular}{ccc}
\includegraphics[width=5cm]{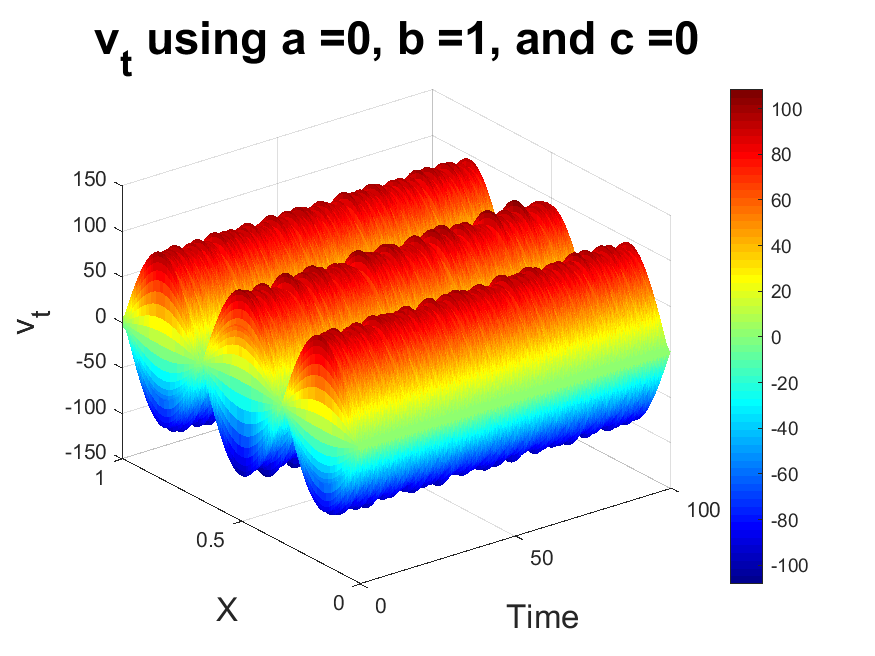} & %
\includegraphics[width=5cm]{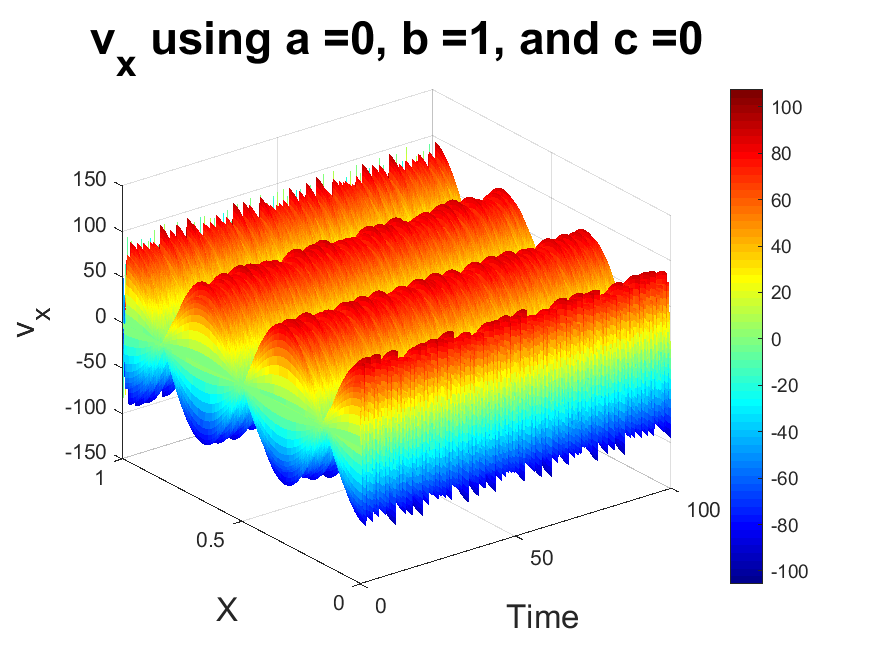} & %
\includegraphics[width=5cm]{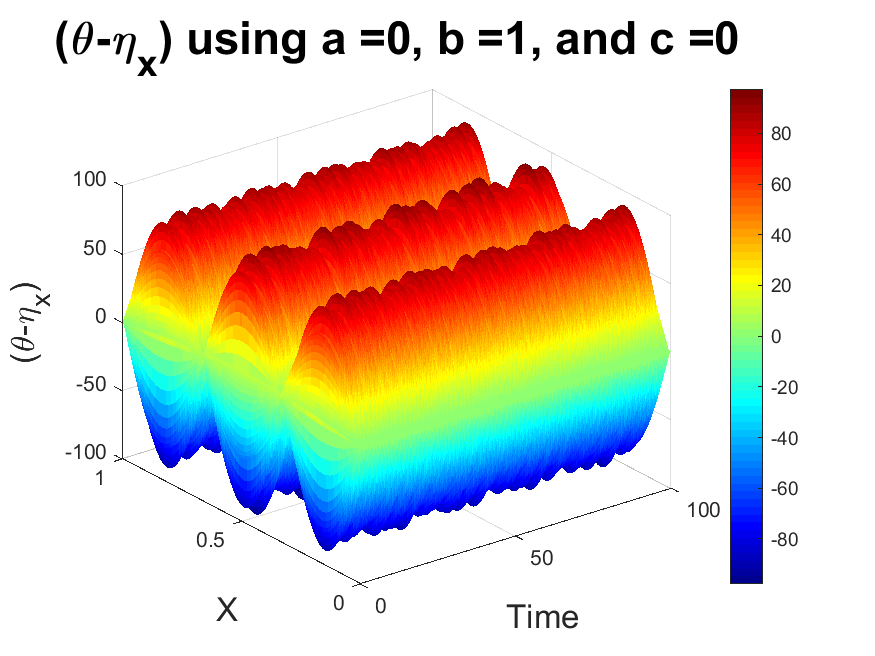} \\ 
\includegraphics[width=5cm]{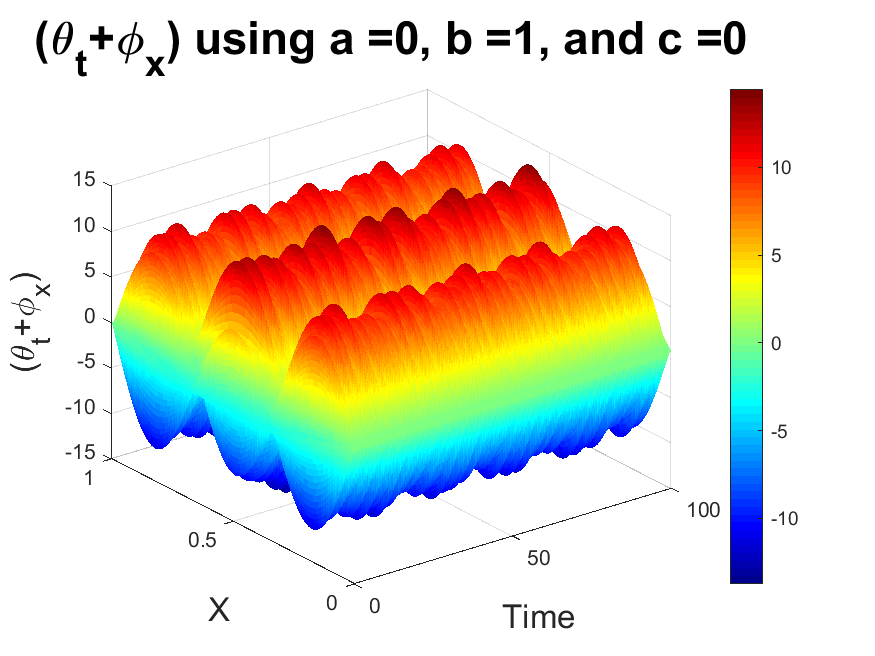} & %
\includegraphics[width=5cm]{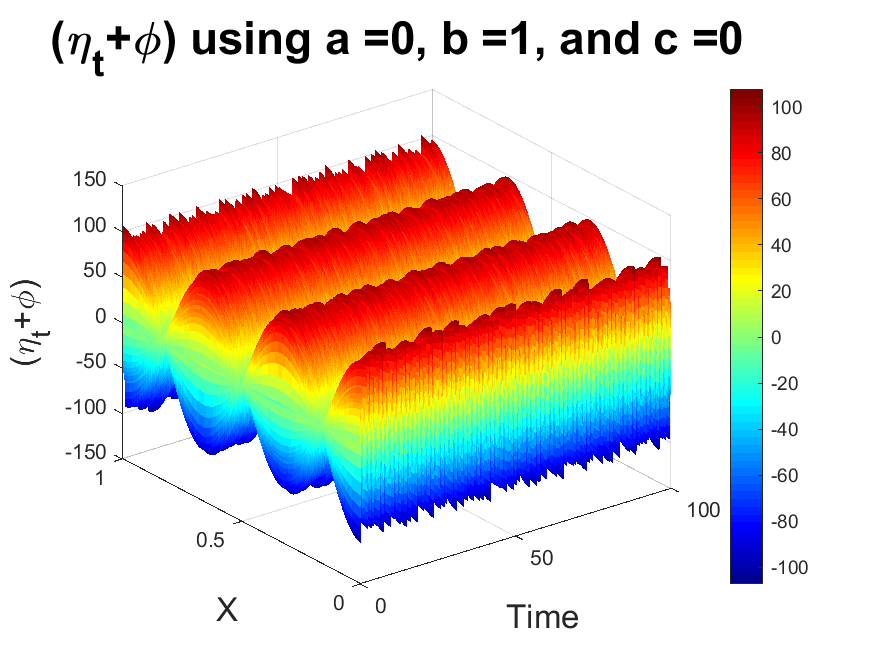} & %
\includegraphics[width=5cm]{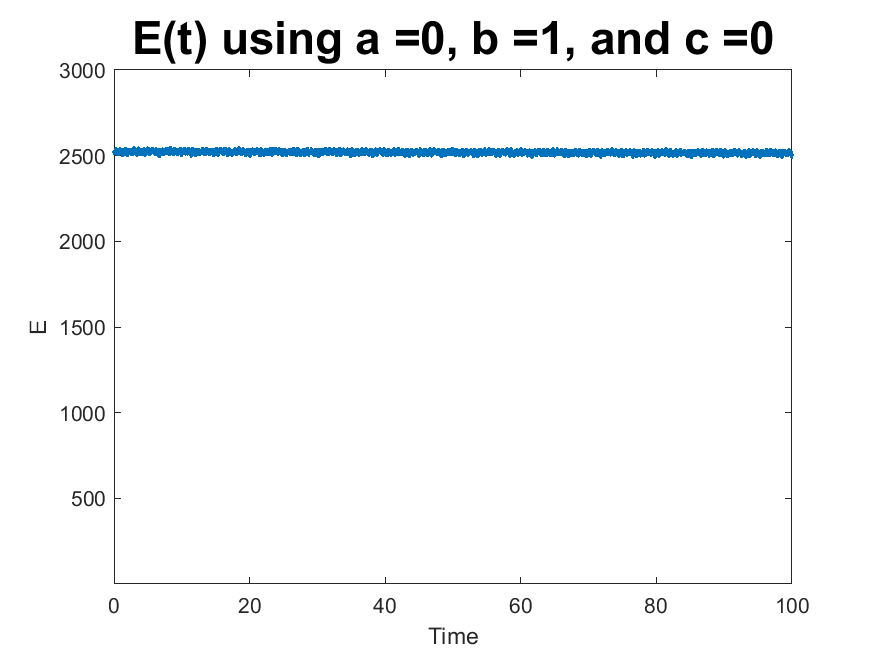} \\ 
&  & 
\end{tabular}
\end{center}
\caption{Space and time numerical solutions when $a=0$, $b=1$ and $c=0$
(open problem).}
\label{a=0b=1c=0}
\end{figure}

%\protect\bibliographystyle{abbrv}
%\protect\bibliographystyle{alpha}
%\bibliography{References}

\end{document}